\documentclass[12pt]{amsart}

\usepackage{graphicx}
\usepackage{amsthm}
\usepackage{amsmath, amssymb}
\usepackage{wrapfig, floatrow}
\swapnumbers

\theoremstyle{plain}
\newtheorem{Thm}{Theorem}

\newtheorem{Coro}[Thm]{Corollary}
\newtheorem{Lem}[Thm]{Lemma}

\theoremstyle{definition}
\newtheorem{Def}[Thm]{Definition}

\begin{document}

\title{Non-uniqueness of high distance Heegaard splittings}

\author{Jesse Johnson}
\address{\hskip-\parindent
        Department of Mathematics \\
        Oklahoma State University \\
        Stillwater, OK 74078 \\
        USA}
\email{jjohnson@math.okstate.edu}


\thanks{This project was supported by NSF Grant DMS-1006369}

\begin{abstract}
Kevin Hartshorn showed that if a three-dimensional manifold $M$ admits a Heegaard surface $\Sigma$ with Hempel distance $d$ then every incompressible surface in $M$ has genus at least $\frac{d}{2}$. Scharlemann-Tomova generalized this, proving that in such a manifold, every other Heegaard surface for $M$ of genus $g' < \frac{d}{2}$ is a stabilization of $\Sigma$. In the present paper, we show that Hartshorn's bound is sharp and Scharlemann-Tomova's bound is very close to sharp. In particular, for every pair of integers $g \geq 2, d \geq 2$, we construct a three-manifold $M$ with a genus $g$, distance $d$ Heegaard splitting and an incompressible surface of genus $\frac{d}{2}$. We also construct, for every $d \geq 4$, a three-manifold with a genus $g$, distance $d$ Heegaard surface $\Sigma$ and a second Heegaard surface with genus $g' = \frac{1}{2} d + g - 1$ that is not a stabilization of $\Sigma$.
\end{abstract}

\maketitle

A \textit{handlebody} is a three-dimensional manifold homeomorphic to a regular neighborhood of a graph embedded in $\mathbf{R}^3$. A \textit{Heegaard splitting} $(\Sigma, H^-, H^+)$ of a compact, connected, closed, orientable three-dimensional manifold $M$ is a decomposition of $M$ into handlebodies $H^-, H^+ \subset M$ whose intersection is precisely their common boundary surface $\Sigma = \partial H^- = \partial H^+$. 

The \textit{(Hempel) distance} $d(\Sigma)$ of a Heegaard splitting is a measure of its complexity introduced by Hempel~\cite{Hempel} that will be defined carefully below. Work by Namazi-Souto~\cite{Namazi-Souto}, Hass-Thompson-Thurston~\cite{Hass-Thompson-Thurston} and others has shown that the distance of a Heegaard splitting is closely related to the large-scale (hyperbolic) geometric structure of the ambient three-manifold. However, the distance $d(\Sigma)$ is a topological concept and thus deeply connected to the topology of $M$. 

Most notably, Hartshorn~\cite{Hartshorn} showed that if $M$ admits a high distance Heegaard splitting $(\Sigma, H^-, H^+)$ then every incompressible surface in $M$ has genus at least $\frac{1}{2} d(\Sigma)$. The first result in this paper shows that Hartshorn's bound is sharp.

\begin{Thm}
\label{thm:main1}
For every pair of positive integers $d \geq 2$, $g \geq 2$ with $d$ even, there is a compact, connected, closed, orientable three-manifold $M$ with both a genus $g$ Heegaard surface $\Sigma$ such that $d(\Sigma) = d$, and a separating, two-sided, closed, embedded, incompressible surface of genus $\frac{1}{2}d$.
\end{Thm}

Scharlemann-Tomova~\cite{Scharlemann-Tomova} showed, moreover, that every other genus $g$ Heegaard surface $\Sigma'$ for $M$ with genus strictly less than $\frac{1}{2}d(\Sigma)$ is a stabilization of $\Sigma$, i.e.\ the result of attaching a number of unknotted handles to $\Sigma$. This is equivalent to the statement that any irreducible Heegaard splitting for $M$ is either isotopic to $\Sigma$ or has genus greater than or equal to $\frac{1}{2} d(\Sigma)$. In fact, in all previously known examples of manifolds with multiple irreducible Heegaard splittings, all of the Heegaard splittings have distance at most three. This begs the question: Can Scharlemann-Tomova's result be strengthened to show that if $d(\Sigma)$ is sufficiently high then $\Sigma$ is the only irreducible Heegaard splitting of $M$? The second result in this paper gives a negative answer to this question.

\begin{Thm}
\label{thm:main2}
For every pair of positive integers $d \geq 4$, $g \geq 2$ with $d$ even, there is a three-manifold $M$ with both a genus $g$ Heegaard surface $\Sigma$ such that $d(\Sigma) = d$ and a second Heegaard surface $S$ of genus $\frac{1}{2}d + g - 1$ such that $S$ is not a stabilization of $\Sigma$.
\end{Thm}

Note that while the incompressible surface constructed in Theorem~\ref{thm:main1} shows that Hartshorn's Theorem is sharp, the genus of the alternate Heegaard surface constructed in Theorem~\ref{thm:main2} is slightly higher than the bound $\frac{d}{2}$ given by Scharlemann-Tomova. It thus remains an open problem to determine whether the Scharlemann-Tomova's bound is sharp.

While the statements and proofs of these Theorems are completely topological, the intuition behind this approach comes from the hyperbolic geometry picture, and this is explained in Section~\ref{sect:intuition}. We define Hempel distance and subsurface projection distance in Section~\ref{sect:definitions}, then use these ideas in Section~\ref{sect:geodesics} to construct certain precisely controlled geodesics in the curve complex, similar to geodesics constructed by Ido-Jang-Kobayashi~\cite{IdoJangKobayashi}. We prove a number of results about subsurface projections of handlebody sets in Section~\ref{sect:handlebodyproj}, then combine these results with the previously constructed geodesics to construct a three-manifold $M$ with a Heegaard splitting $\Sigma$ with a specified value of $d(\Sigma)$ in Section~\ref{sect:splitting1}.

Next, we turn to the alternate surfaces. In Section~\ref{sect:flat}, we describe how to construct an embedded surface in a handlebody based on a path in the curve complex. In Section~\ref{sect:incompressible}, we show that appropriate conditions on the path used to construct such a surface will guarantee that it is incompressible, proving Theorem~\ref{thm:main1}. Then in Section~\ref{sect:splitting2}, we generalize the construction to produce a strongly irreducible Heegaard surface, proving Theorem~\ref{thm:main2}.

The constructions used to produce the incompressible surfaces and the alternate Heegaard surfaces below are very similar. In fact, the Heegaard surface is essentially built from two incompressible surfaces, attached together by a construction that is sometimes called a {\it toggle}. (This reinforces a common theme that strongly irreducible Heegaard surfaces are very close to incompressible.) However, the portion of the curve complex path that corresponds to this toggle has to satisfy a certain condition that makes it impossible to directly construct an incompressible surface in the same way. As a result, the three-manifolds that contain low genus incompressible surfaces and those that contain low genus alternate Heegaard surfaces appear to be different. So, it remains an open problem to determine whether the manifolds that admit alternate unstabilized Heegaard splittings also contain low genus (or any) incompressible surfaces.

Because the alternate surfaces constructed below correspond to geodesics in the curve complex between the disk sets for the original Heegaard surface $\Sigma$, it is conceivable that if there are multiple geodesics between the handlebody sets (or multiple paths that are somehow locally geodesic) then these would correspond to non-isotopic alternate Heegaard splittings. Note that delta-hyperbolicity of the curve complex~\cite{Masur-Minsky1} rules out the possibility of having distinct geodesics far away from each other, but does not rule out having multiple parallel geodesics within the delta hyperbolicity bounds.

Finally, we note that while the original Heegaard splittings in these examples have high Hempel distance, the alternate Heegaard surfaces all have Hempel distance two. This is suggestive of Schleimer's result~\cite{Schleimer} that for every three-manifold there is a constant $c$ such that every Heegaard surface of genus greater than $c$ has distance at most two. The only currently known examples of manifolds with more than one Heegaard surface of distance strictly greater than two are the Berge-Scharlemann manifolds~\cite{Berge-Scharlemann, Scharlemann}, each of which has two distinct, distance three Heegaard splittings. These examples can probably be generalized to higher genus, but it remains open whether a manifold can have two distinct Heegaard splittings, both with distance strictly greater than three.

I thank, Yeonhee Jang, Tsuyoshi Kobayashi, Saul Schleimer and Kazuto Takao for a number of valuable conversations that led to this work.

\section{Background and intuition}
\label{sect:intuition}

A major theme in research on three-dimensional manifolds over the last few decades has been the connections between three-dimensional topology, three-dimensional geometry and two-dimensional Teichmuller theory initiated by Thurston's geometrization program~\cite{Thurston}. A relatively simple illustration of these connections is the geometrization of surface bundles, in which a topological property on a three-maifold $M$ (that $M$ is a surface bundle, has no incompressible tori and is not a small Seifert fibered space) implies a geometric property (that $M$ has a complete hyperbolic metric) that can be understood in terms of Teichmuller theory on an embedded surface. (The proof~\cite{Otal} involves finding a fixed point of an automorphism of a space closely related to the Teichmuller space for a leaf in the bundle structure.)

In recent years, the connections between Teichmuller theory and three-dimensional geometry and topology have become much more precise, particularly in terms of Heegaard splittings. For example Namazi-Souto have constructed very accurate approximations of the hyperbolic structures of manifolds with certain types of Heegaard splittings~\cite{Namazi-Souto}. They show that if the Hempel distance for the Heegaard splitting (which will be defined below) is sufficiently high and satisfies certain other conditions then the hyperbolic metric on the manifold has relatively narrow cross sections parallel to the Heegaard surface, but a long diameter perpendicular to these cross sections. This is shown schematically in Figure~\ref{fig:NamaziSoutoManifold}.
\begin{figure}[htb]
  \begin{center}
  \includegraphics[width=4.5in]{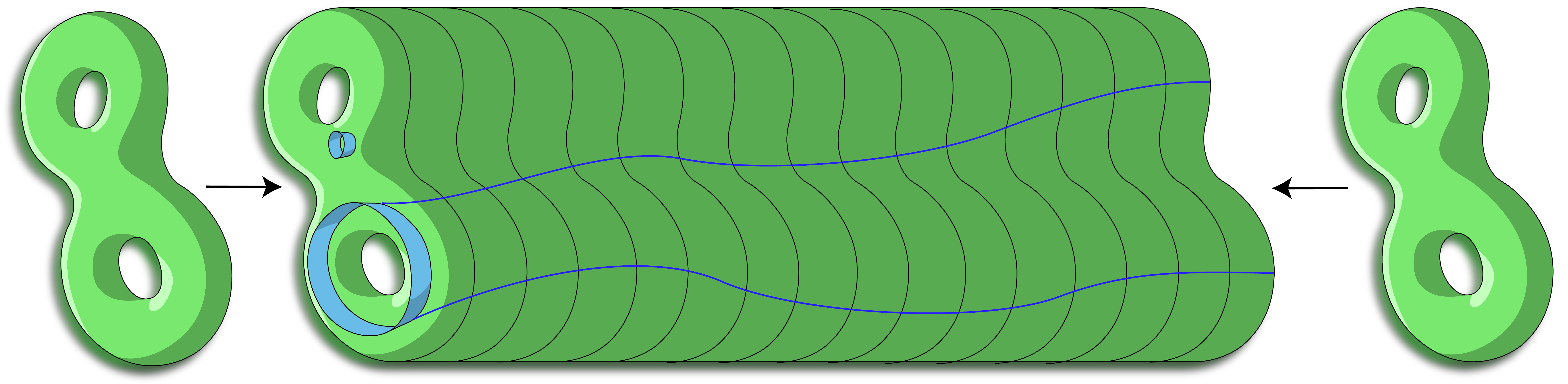}
  \caption{The geometric structure of a manifold with a high distance Heegaard splitting.}
  \label{fig:NamaziSoutoManifold}
  \end{center}
\end{figure}

The Hempel distance is defined in terms of the curve complex for the Heegaard surface, so Namazi-Souto's work strengthens the connection between Teichmuller theory and three-dimensional geometry. Meanwhile, a number of results have strengthened the connection between Teichmuller theory and three-dimensional topology by showing that given a Heegaard surface with sufficiently high Hempel distance, the ambient manifold has no incompressible tori (Hempel~\cite{Hempel}), no low genus incompressible surfaces (Hartshorn~\cite{Hartshorn}) or no low genus alternate Heegaard surfaces (Scharlemann-Tomova~\cite{Scharlemann-Tomova}). (These results all have roots in work of Kobayashi~\cite{Kobayashi} that predates Hempel's definition of distance.)

In this context, the direct connection between three-dimensional geometry and topology was developed last, with the (completely geometric) proof by Hass-Thompson-Thurston that certain high distance Heegaard splittings need to be stabilized many times before there is an isotopy that interchanges their complementary handlebodies~\cite{Hass-Thompson-Thurston}. They construct a hyperbolic structure on the ambient manifold $M$ similar to that of Namazi-Souto, then analyze how a harmonic representative of an embedded surface behaves in such a metric. Their proof suggests a very nice geometric interpretation of the previously mentioned topological results~\cite{Hartshorn, Scharlemann-Tomova}. Namely, one might expect these results to be true because any surface that is not parallel to the narrow cross sections of $M$ must stretch the long way across the geometric structure and thus have very large area, as suggested (again schematically) by the blue surface in Figure~\ref{fig:NamaziSoutoManifold}. A careful application of the Gauss-Bonnet Theorem turns this lower bound on area into a lower bound on genus.

This idea can be generalized using the notion of subsurface projection, which will be defined below. Roughly speaking, the geometric interpretation of subsurface distance in Minsky's approach to the ending lamination conjecture~\cite{Minsky} suggests that in a hyperbolic manifold, a high distance subsurface will correspond to a large volume concentrated around the subsurface. One can think of this as a sort of bubble in the hyperbolic metric as in the upper half of Figure~\ref{fig:Bubble}. This is very difficult to make precise geometrically, but the topological interpretation is fairly straightforward: Because this bubble is long in one direction, but has narrow cross sections parallel to the high distance subsurface, any low genus Heegaard surface should be forced to intersect the bubble along the narrow cross section, as indicated by the curve in Figure~\ref{fig:Bubble}. Based on this idea, Yair Minsky, Yoav Moriah and the author showed that if a Heegaard splitting has a high distance subsurface then every other low genus Heegaard surface must have a subsurface parallel to this high distance subsurface (though not necessarily to the rest of the Heegaard surface)~\cite{J-Minsky-Moriah}.
\begin{figure}[htb]
  \begin{center}
  \includegraphics[width=3.5in]{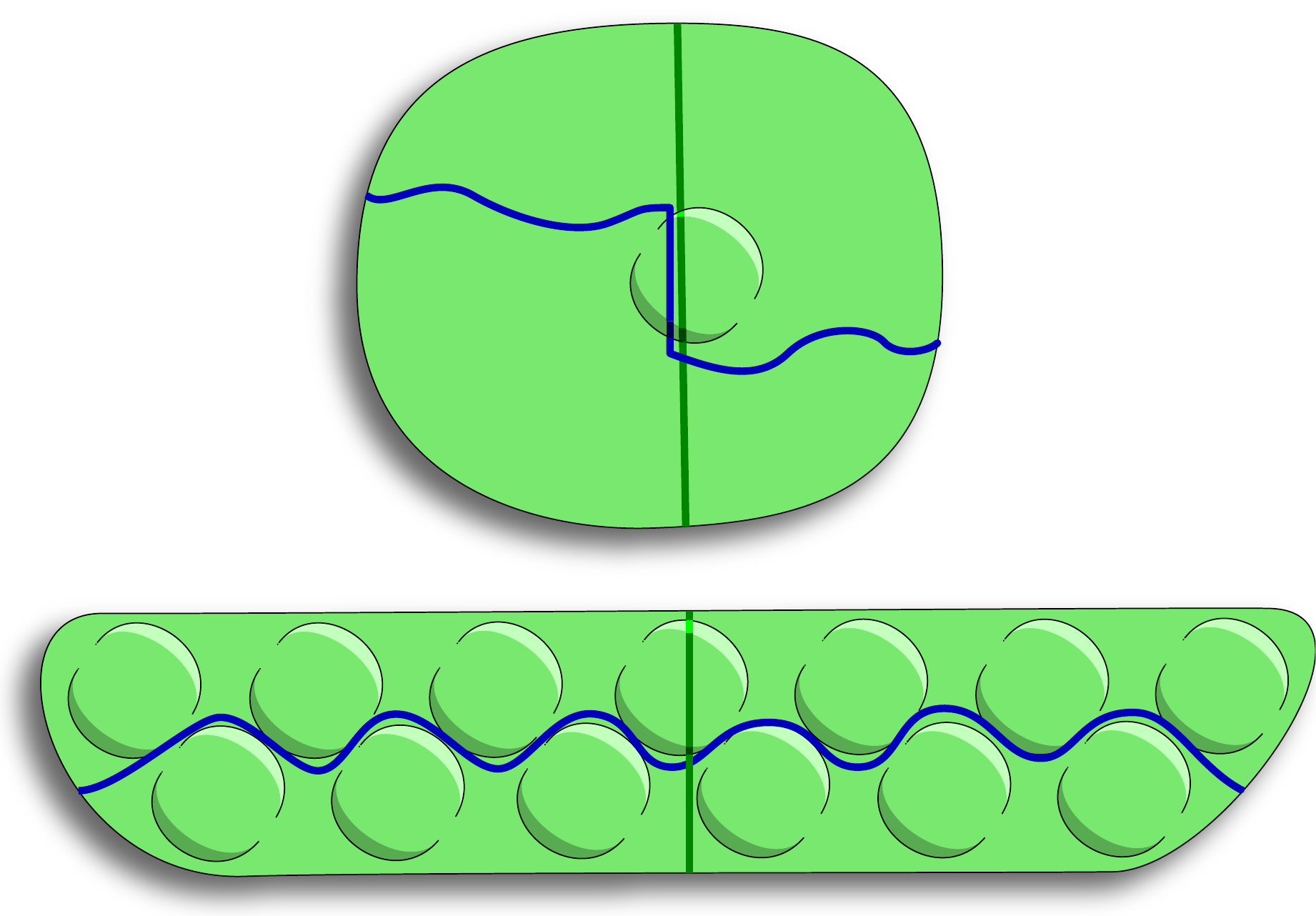}
  \caption{High distance subsurfaces imply ``bubbles" of high volume in the hyperbolic metric.}
  \label{fig:Bubble}
  \end{center}
\end{figure}

In the present paper, we extend this idea by constructing a Heegaard splitting with a large number of distinct (and overlapping) high distance subsurfaces. Our construction is very similar to the construction for extending geodesics described recently by Birman-Menasco~\cite{Birman-Menasco} and Ido-Jang-Kobayashi~\cite{IdoJangKobayashi} and closely relate to methods for constructing Heegaard splittings with specified distances introduced by Qiu-Zou-Guo~\cite{QiuZouGuo}. All of these methods rely heavily on ideas developed in Saul Scleimer's notes on the curve complex~\cite{SchleimerNotes} and Masur-Schleimer's work on holes in the curve complex~\cite{Masur-Schleimer}.

One can visualize the examples that we construct below geometrically as a manifold consisting of two parallel rows of hyperbolic bubbles as in the lower half of Figure~\ref{fig:Bubble}. The low genus Heegaard surface containing these subsurfaces cuts the manifold in half the short way, as indicated by the vertical curve in the Figure. The alternate incompressible or Heegaard surface can be thought of as weaving between these bubbles the long way across the manifold, as suggested by the horizontal curve in Figure~\ref{fig:Bubble}. The geometric intuition suggests that this second Heegaard surface should be caught between the bubbles since any isotopy from it to the original Heegaard surface would have to cross one of the bubbles the long way. We prove (completely topologically) that this is in fact the case: Given the appropriate conditions on the original Heegaard surface $\Sigma$, the alternate surface $S$ will be incompressible or strongly irreducible, and thus not isotopic to a stabilization of $\Sigma$.

\section{Distance and subsurface projection}
\label{sect:definitions}

The \textit{curve complex} for a compact, connected, orientable surface $\Sigma$ is the simplicial complex $\mathcal{C}(\Sigma)$ in which each vertex represents an isotopy class of essential simple closed curves in $\Sigma$ and each simplex represents a set of isotopy classes with pairwise disjoint representatives. In particular, if two isotopy classes are represented by disjoint loops then there is an edge between the corresponding vertices of $\mathcal{C}(\Sigma)$.

We will endow the vertex set of $\mathcal{C}(\Sigma)$ with the structure of a metric space by thinking of each edge in $\mathcal{C}(\Sigma)$ as having length one. So, for vertices $v,w \in \mathcal{C}(\Sigma)$, we define the distance $d(v,w)$ to be the number of edges in the shortest path from $v$ to $w$ in $\mathcal{C}(\Sigma)$. It is a relatively simple exercise to show that $\mathcal{C}(\Sigma)$ is connected (so $d(v,w)$ is a finite integer for each pair $v,w$) and a much harder exercise to show that $\mathcal{C}(\Sigma)$ has infinite diameter~\cite{Masur-Minsky1, SchleimerNotes}.

If $\Sigma$ is the boundary of a handlebody $H$ then we define the \textit{handlebody set} of $H$ to be the set $\mathcal{H} \subset \mathcal{C}(\Sigma)$ consisting of all vertices whose representatives bound disks in $H$. Given a Heegaard splitting $(\Sigma, H^-, H^+)$, the surface $\Sigma$ is the boundary of two different handlebodies, so this defines a handlebody set $\mathcal{H}^-$ for $H^-$ and a second set $\mathcal{H}^+$ for $H^+$, as in Figure~\ref{fig:distance}. 

\begin{Def}
The \textit{(Hempel) distance} of $(\Sigma, H^-, H^+)$ is
$$d(\Sigma) = d(\mathcal{H}^-, \mathcal{H}^+) = \min\{d(v,w)\ |\ v \in \mathcal{H}^-, w \in \mathcal{H}^+\}.$$
\end{Def}
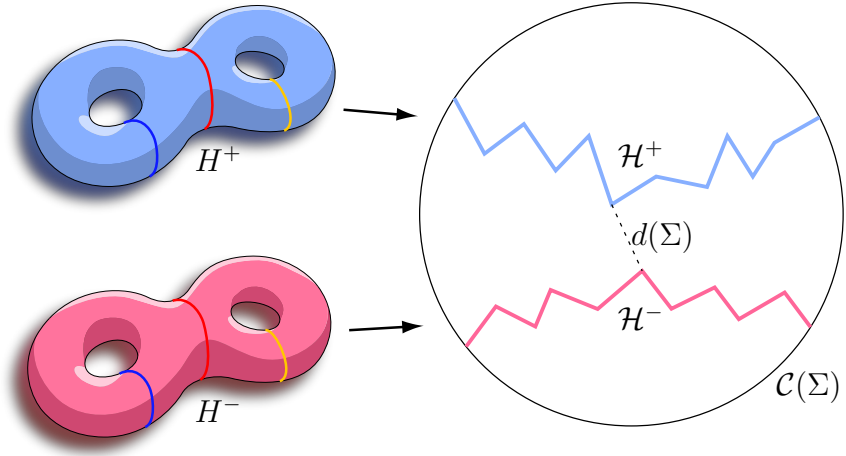
\begin{figure}[htb]
  \begin{center}
  \includegraphics[width=4.5in]{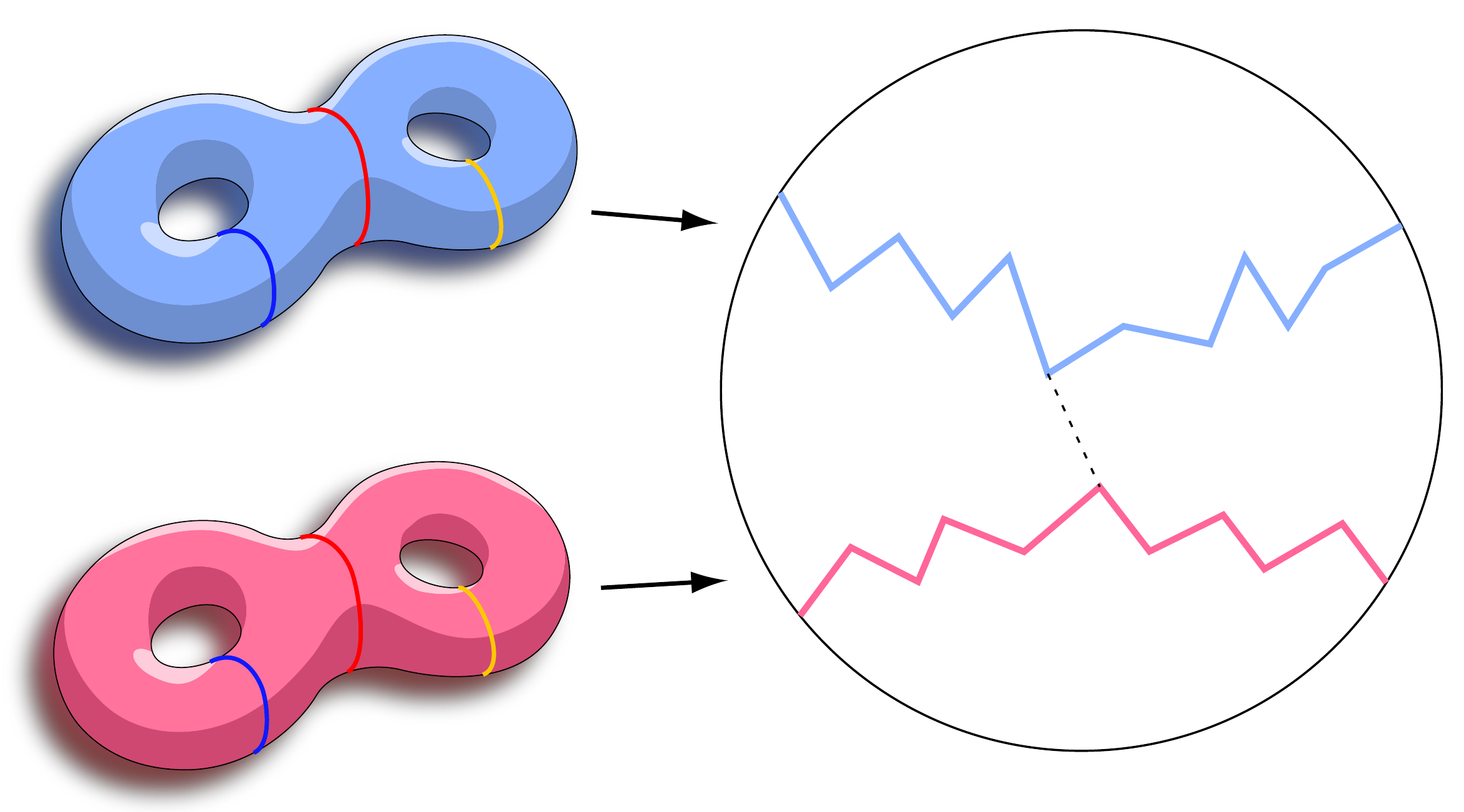}
  \put(-250,110){$H^+$}
  \put(-250,15){$H^-$}
  \put(-90,112){$\mathcal{H}^+$}
  \put(-90,50){$\mathcal{H}^-$}
  \put(-85,82){$d(\Sigma)$}
  \put(-30,25){$\mathcal{C}(\Sigma)$}
  \caption{Hempel distance $d(\Sigma)$ of a Heegaard splitting $(\Sigma, H^-, H^+)$.}
  \label{fig:distance}
  \end{center}
\end{figure}

For a compact surface $F$ with boundary, we define the \textit{arc and curve complex} $\mathcal{AC}(F)$ similarly, except that each vertex represents either an isotopy class of essential (including not boundary parallel) simple closed curves, or an isotopy class of essential properly embedded arcs. In this paper, we will consider the arc and curve complexes for different subsurfaces of a closed surface $\Sigma$. 

In particular, we will say that a subsurface $F$ of $\Sigma$ is \textit{essential} if each loop of $\partial F$ is either essential in $\Sigma$ or contained in a component of $\partial \Sigma$. For each essential loop $\ell \subset \Sigma$, we will define the \textit{subsurface projection} $p_F(\ell)$ to be the set of vertices in $\mathcal{AC}(F)$ corresponding to components of $\ell \cap F$ after $\ell$ has been isotoped to intersect $\partial F$ minimally. The projections of two loops into a genus-one subsurface are shown in Figure~\ref{fig:subsurfproj}.

Note that the projection of a loop into $\mathcal{AC}(F)$ is a (possibly empty) set of vertices rather than a single vertex. However, because $\ell$ is embedded, the components of $\ell \cap F$ will be pairwise disjoint, so $p_F(\ell)$ will span a simplex in $\mathcal{AC}(F)$ (if $p_F(\ell)$ is not the empty set).
\begin{figure}[htb]
  \begin{center}
  \includegraphics[width=4.5in]{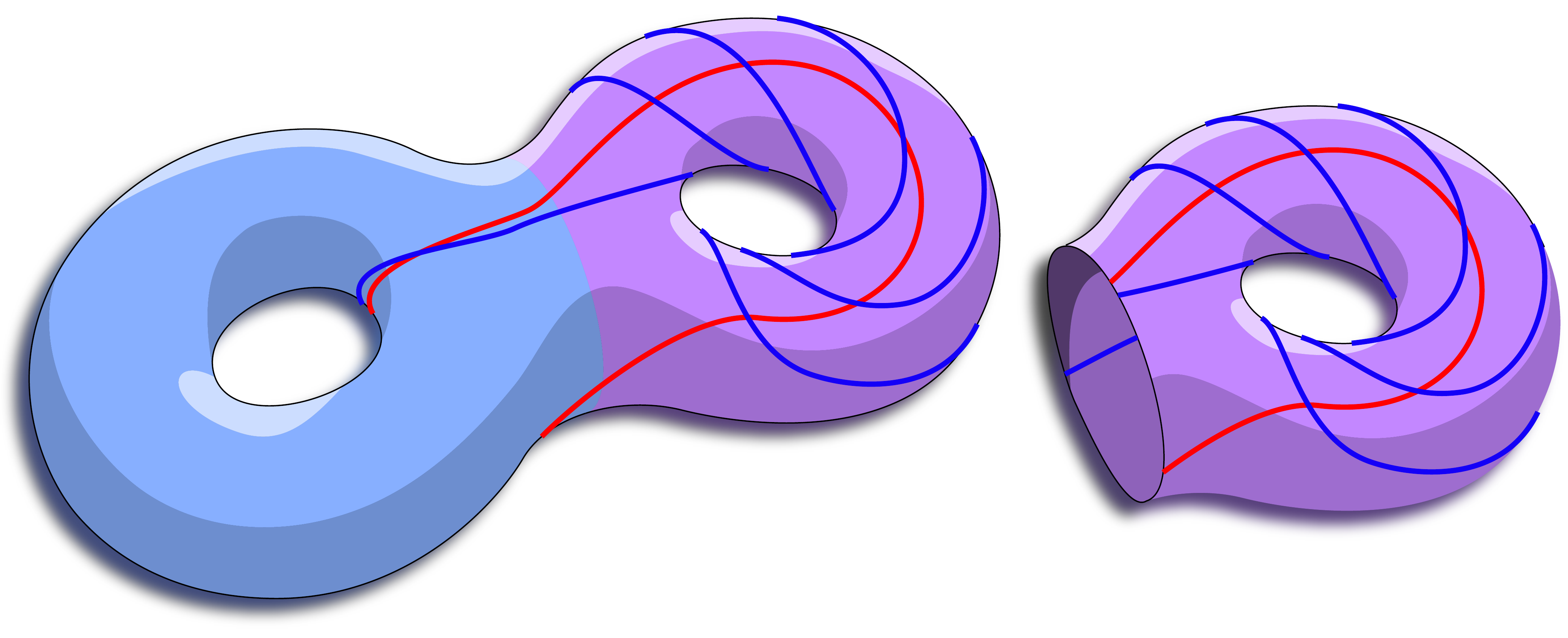}
  \put(-200,30){$\Sigma$}
  \put(-50,10){$F$}
  \caption{Subsurface projections of loops.}
  \label{fig:subsurfproj}
  \end{center}
\end{figure}

Given sets $A, B \subset \mathcal{C}(\Sigma)$, define $d_F(A,B)$ to be the distance in $\mathcal{AC}(F)$ from $p(A)$ (the union of the projections of vertices in $A$) to $p(B)$. This is only well defined if both projections are not empty, so we will define $d_F(A,B) = \infty$ otherwise. The key to our use of subsurface projections will be the following Lemma:

\begin{Lem}
\label{lem:mustmissF}
If $d_F(A,B) = m$ then every path in $\mathcal{C}(\Sigma)$ of length $n < m$ from $A$ to $B$ contains a loop disjoint from $F$.
\end{Lem}

\begin{proof}
First note that if either projection is empty then either every loop in $A$ or every loop in $B$ is disjoint from $F$, so every path from $A$ to $B$ must start or end in a loop disjoint from $F$ and the proof is complete.

Otherwise, consider $A, B \subset \mathcal{C}(\Sigma)$ such that $d_F(A,B) = m > n$ is finite. Let $\ell_0,\dots,\ell_n$ be a path in $\mathcal{C}(\Sigma)$ from $A$ to $B$ and assume for contradiction that no $\ell_i$ can be isotoped out of $F$. Then $p_F(\ell_i)$ contains at least one vertex $v_i \in \mathcal{AC}(F)$. Because the loops $\ell_i$, $\ell_{i+1}$ are disjoint, their intersections with $F$ are disjoint, so $v_i$, $v_{i+1}$ either are equal or bound an edge in $\mathcal{AC}(F)$. The sequence of vertices $v_0,\dots,v_n$ thus defines a length-$n$ path from $p_F(A)$ to $p_F(B)$, contradicting the assumption that $d_F(A,B) = m > n$. This contradiction implies that some loop $\ell_i$ must be disjoint from $F$.
\end{proof}

\section{Geodesics and handlebody sets}
\label{sect:geodesics}

We will use Lemma~\ref{lem:mustmissF} to construct flexible geodesics similar to the ones constructed by Ido-Jang-Kobayashi~\cite{IdoJangKobayashi}. Let $k > 0$ be a positive integer whose role will be explained below and let $d \geq 2$ be any even integer, the length of the desired path. Let $\Sigma$ be a compact, connected, closed, orientable, genus $g$ surface and let $\ell_0$ be a separating, essential, simple closed curve in $\Sigma$. Moreover, assume that the closure of one component of the complement $\Sigma \setminus \ell_0$ is a once-punctured torus $F_0 \subset \Sigma$.

Let $\ell_1$ be a non-separating, essential, simple closed curve in $F_0$, as in Figure~\ref{fig:flexipath}. Let $F_1$ be the complement in $\Sigma$ of an open regular neighborhood of $\ell_1$. This will be a genus $g-1$ subsurface with two boundary loops parallel to $\ell_1$. Then the loop $\ell_0$ defines a vertex in the arc and curve complex $\mathcal{AC}(F_1)$. Since this complex has infinite diameter, we can choose an arc $\alpha_2 \subset F_1$, with one endpoint in each boundary loop of $F_1$, such that $d_{F_1}(\ell_0, \alpha_2) > k + d$. A closed regular neighborhood $F_2 \subset \Sigma$ of the union $\ell_1 \cup \alpha_2$ will be a once-punctured torus and we will define $\ell_2$ to be the boundary of $F_2$, as in Figure~\ref{fig:flexipath}. Similarly, $\ell_1$ is an essential loop in $F_2$ and the arc and curve complex for $F_2$ has infinite diameter, so we can choose a loop $\ell_3 \subset F_2$ such that $d_{F_2} (\ell_1, \ell_3)$ is arbitrarily high. For any such loop $\ell_3$ that is not isotopic to $\ell_1$, the projection of $\ell_3$ to $F_1$ is (multiple copies of) the arc $\alpha_4$, so $d_{F_1}(\ell_0, \ell_3) > k + d$.
\begin{figure}[htb]
  \begin{center}
  \includegraphics[width=3.5in]{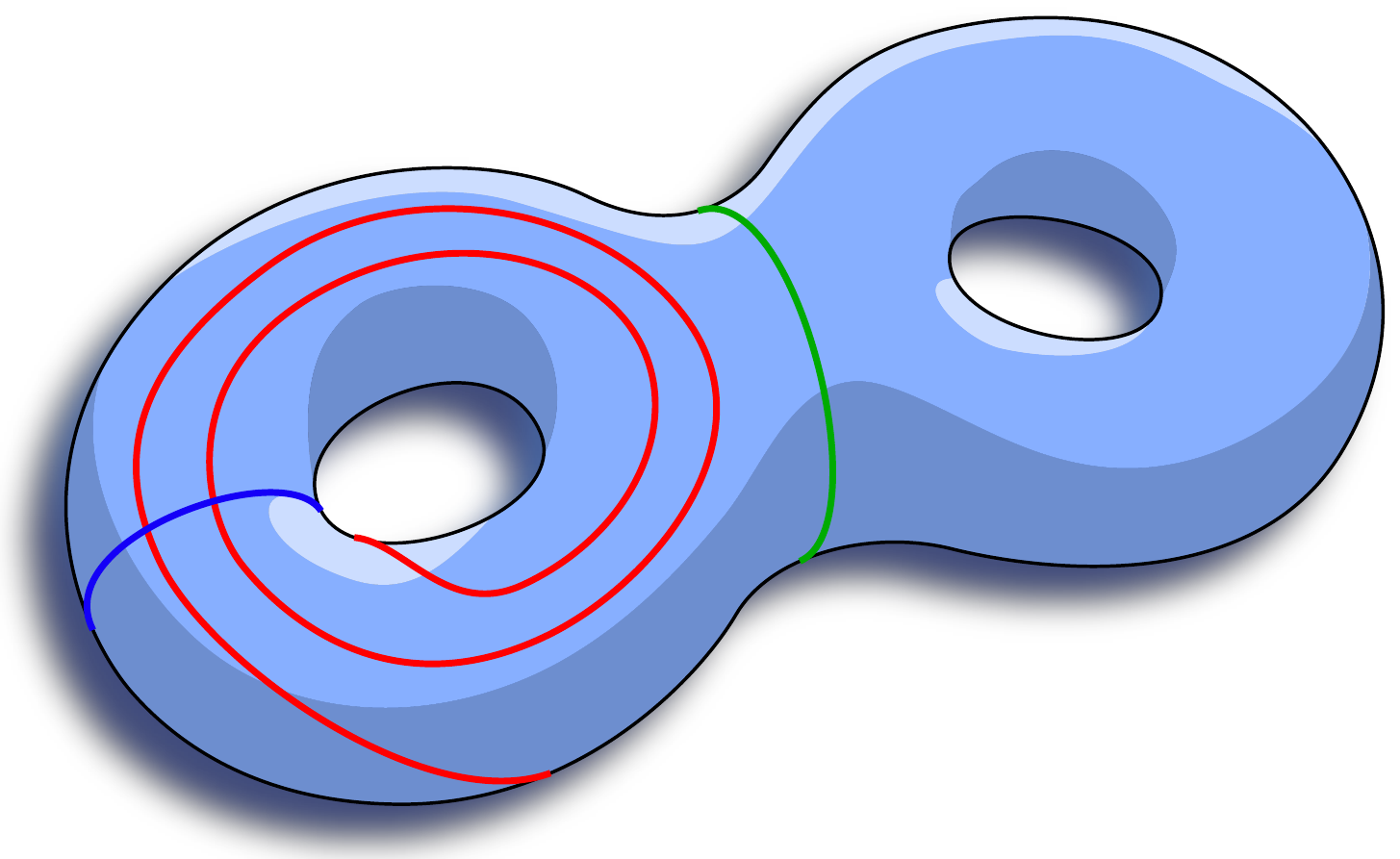}
  \put(-247,35){$\ell_1$}
  \put(-105,40){$\ell_2$}
  \put(-150,5){$\ell_3$}
  \caption{Loops forming part of a flexipath.}
  \label{fig:flexipath}
  \end{center}
\end{figure}

Let $F_3$ be the complement in $\Sigma$ of a regular open neighborhood of $\ell_3$. Like $F_1$, this will be a genus $g-1$ subsurface with two boundary loops parallel to $\ell_3$. Note that the projection of $\ell_0$ into $\mathcal{AC}(F_3)$ is a simplex (a collection of pairwise-disjoint arcs), while each of $\ell_1$ and $\ell_2$ defines a single vertex in $F_3$. Since $\mathcal{AC}(F_3)$ has infinite diameter, we can choose a properly embedded, non-separating arc $\alpha_4 \subset F_3$ such that $d_{F_3}(\ell_1 \cup \ell_2 \cup \ell_3, \ell_4) > k + d + 1$. Let $\ell_4$ be the boundary of a closed regular neighborhood $F_4$ of $\ell_3 \cup \alpha_4$. Then by the triangle inequality, $d_{F_3}(\ell_1 \cup \ell_2 \cup \ell_3, \alpha_4) > k + d$. We can then repeat the process by choosing a loop $\ell_5 \subset F_4$ so that $d_{F_4}(\ell_3, \ell_5)$ is arbitrary, and so on.

\begin{Lem}
\label{lem:flexipath1}
For every pair of integers $k > 0$, $d \geq 2$ with $d$ even, there is a path of loops $\ell_0,\ldots,\ell_d$ such that
\begin{itemize}
\item For each even $i$, $\ell_i$ bounds a once-punctured torus containing $\ell_{i-1}$ (if $i > 0$) and $\ell_{i+1}$ (if $i < d$), both of which are non-separating in this torus.
\item For each pair of positive odd integers $j < i \leq d-3$, we have $d_{F_i}(\ell_j, \ell_{i+2}) > k + d$.
\item We have $d_{F_1}(\ell_0, \ell_3) > k + d$ and for every $j < d-1$, we have $d_{F_{d-1}}(\ell_j, \ell_d) > k + d$.
\end{itemize}
\end{Lem}

Note that the second condition in this Lemma are rather asymmetric: It only controls the distances projections of two loops $\ell_j$, $\ell_{i+2}$ in the subsurface $F_i$. We would like to control their distances in any subsurface, and we can do this with a little more work.

\begin{Lem}
\label{lem:flexipath2}
Let $\ell_0,\ldots,\ell_d$ be a path as in Lemma~\ref{lem:flexipath1}. Then for every triple of integers $h < i < j$ such that $i$ is odd, we will have $d_{F_i}(\ell_h, \ell_j) > k$.
\end{Lem}

\begin{proof}
Assume the path $\ell_0,\ldots,\ell_d$ satisfies the conditions described in Lemma~\ref{lem:flexipath1}. Given $h < i < j$ with $i$ odd, note that if $j = i+2$ or if $j = d$ and $i = d-1$ then the condition is satisfied by assumption. Otherwise, we have two cases to consider. 

First, if $j = i + 1$ then $\ell_i$ is the separating loop that bounds a regular neighborhood of $\ell_i \cup \ell_{j+1}$. The loops $\ell_{j+1}$ and $\ell_j$ are disjoint by definition so their projections into $\mathcal{AC}(F_i)$ are distance at most one. By assumption, $d_{F_i}(\ell_h, \ell_{j+1}) = d_{F_i}(\ell_h, \ell_{i+2}) > d + k$ so we have $d_{F_i}(\ell_h, \ell_{j+1}) > d + k - 1 > k$ (since $d \geq 2$).

Otherwise, if $j > i + 2$ then note that for each loop $\ell_{j'}$ with $j' > i$, if $i'$ is the largest odd number strictly less than $j'$ then $d_{F_{i'}}(\ell_i, \ell_{j'})$ is positive, so $\ell_i$ and $\ell_{j'}$ are not isotopic. The complement in $\Sigma$ of $\ell_i$ is an annular neighborhood of $\ell_i$ so every loop in $\Sigma$ that is not isotopic to $\ell_i$ determines a vertex of $\mathcal{AC}(F_i)$. Moreover, the loops $\ell_{i+2}, \ell_{i+3},\ldots,\ell_{j}$ are consecutively disjoint, so they define a path in $\mathcal{AC}(\Sigma)$ of length strictly less than $d$. Thus we conclude (by the triangle inequality) that $d_{F_i}(\ell_h, \ell_j) \geq d_{F_i}(\ell_h, \ell_{i+2}) - d > k$.
\end{proof}

We will give such paths a name:

\begin{Def}
A \textit{$(d,k)$ flexipath} is a path $\ell_0,\ldots,\ell_d$ in $\mathcal{C}(\Sigma)$ such that
\begin{enumerate}
\item For each even $i$, $\ell_i$ bounds a once-punctured torus containing $\ell_{i-1}$ (if $i > 0$) and $\ell_{i+1}$ (if $i < d$).
\item For each $h < i < j$ such that $i$ is even and $h$ and $j$ are odd, $d_{F_i}(\ell_h, \ell_j) > k$.
\end{enumerate}
\end{Def}

In the construction of $\ell_0,\ldots,\ell_d$, we noted that for each even $i \leq d-2$, we could choose $d_i(\ell_{i-1}, \ell_{i+1})$ to be arbitrary. We will say that a $(d,k)$ flexipath is \textit{strict} if $d_i(\ell_{i-1}, \ell_{i+1}) > 6$ for every even $i \leq d - 2$. We will say that the path is \textit{almost strict} if $d_i(\ell_{i-1}, \ell_{i+1}) > 6$ for every odd $i \leq d-2$, except for one even value of $i$ where $\ell_{i-1} \cap \ell_{i+1}$ is a single point. By combining Lemmas~\ref{lem:flexipath1} and~\ref{lem:flexipath2}, we have:

\begin{Coro}
\label{coro:flexipath}
For every pair of integers $k > 0$ and $d \geq 2$ (where $d$ is even), there is a $(d,k)$ flexipath in $\mathcal{C}(\Sigma)$. Moreover, we can choose this path to be strict or almost strict.
\end{Coro}

By now, the role of the integer $k$ should be clear: It keeps track of the subsurface distances, and thus will allow us to use Lemma~\ref{lem:mustmissF} to analyze the path $\ell_0,\ldots,\ell_d$. Note that the path $\ell_0,\dots,\ell_d$ actually determines an infinite family of length-$d$ paths from $\ell_0$ to $\ell_d$: For each even $i \leq d - 2$, the complement in $\Sigma$ of $\ell_{i-1} \cup \ell_{i+1}$ is a once-punctured surface of genus $g - 1$. Since $g \geq 2$, there are infinitely many isotopy classes of loops in this subsurface, each of which is disjoint from both $\ell_{i-1}$ and $\ell_{i+1}$. Thus for any essential loop $\ell'_i$ in this subsurface, the path $\ell_0,\ldots,\ell_{i-1},\ell'_i,\ell_{i+1},\ldots,\ell_d$ is also a path of length $d$. We can do the same for any even index, creating infinitely many distinct geodesics in the curve complex that contain all the odd-index vertices of the original path, as in Figure~\ref{fig:flexipathincc}. This is the reason for calling these flexipaths.
\begin{figure}[htb]
  \begin{center}
  \includegraphics[width=4.5in]{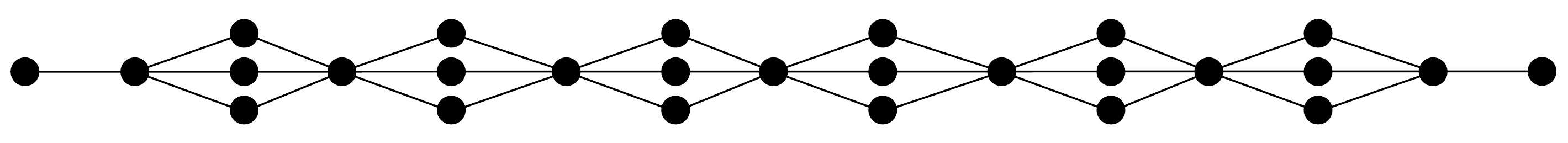}
  \caption{Multiple possible geodesics following a flexipath in the curve complex.}
  \label{fig:flexipathincc}
  \end{center}
\end{figure}

So our original path $\ell_0,\ldots,\ell_d$ will never be the unique geodesic from $\ell_0$ to $\ell_d$, but by focusing on the loops with odd indices, we can get pretty close.

\begin{Lem}
If $\ell_0,\ldots,\ell_d$ is a $(d,k)$ flexipath then every path in $\mathcal{C}(\Sigma)$ of length less than or equal to $k$ from $\ell_0$ to $\ell_d$ passes through every $\ell_i$ for $i$ odd. In particular, if $k \geq d$ then every $(d,k)$ flexipath is a geodesic from $\ell_0$ to $\ell_d$.
\end{Lem}

\begin{proof}
Because the path is a $(d,k)$ flexipath, for each odd $i$ between 0 and $d$, we have $d_{F_i}(\ell_0,\ell_d) > k$. By Lemma~\ref{lem:mustmissF}, this implies that every path of length less than or equal to $k$ from $\ell_0$ to $\ell_d$ must pass through a vertex representing a loop disjoint from $F_i$. However, the complement in $\Sigma$ of $F_i$ is, by construction, an open annular neighborhood of $\ell_i$. Thus the only essential loop disjoint from $F_i$ is $\ell_i$ (up to isotopy) and we conclude that $\ell_i$ is a vertex in the path from $\ell_0$ to $\ell_d$.

For $k \geq d$, we immediately know that our path contains the $\frac{d}{2}$ loops $\ell_1,\ell_3,\ldots,\ell_{d-1}$, though we don't immediately know what order they appear in. However, no two loops $\ell_i$, $\ell_j$ are disjoint for $i \neq j$, both odd, so no such loops $\ell_i$, $\ell_j$ can appear consecutively in a path. This implies that the total number of loops in the path is at least $d$. Since no path from $\ell_0$ to $\ell_d$ can be shorter than length $d$, we conclude that the flexipath $\ell_0,\ldots,\ell_d$ is a geodesic.
\end{proof}

\section{Projections of handlebody sets}
\label{sect:handlebodyproj}
Before we begin to construct the three-manifolds promised in Theorems~\ref{thm:main1} and~\ref{thm:main2}, we will need two technical Lemmas.

\begin{Lem}
\label{lem:diskprojection}
Let $H$ be a handlebody and $F \subset \partial H$ a subsurface such that for every essential disk $D \subset H$, the intersection of $\partial D \cap F$ contains at least three arcs. Then the projection of the disk set for $H$ into $\mathcal{AC}(F)$ has diameter at most four.
\end{Lem}

\begin{proof}
Let $F' \subset \partial H$ be the closure of the complement $\partial H \setminus F$. Note that if $\partial D$ intersects both $F$ and $F'$ then the number of arcs or intersection is the same in both subsurfaces. Let $D \subset H$ be a compressing disk for $H$ that intersects $F'$ (or, equivalently, $F$) minimally among all essential disks in $H$. We will show that for any other disk $D' \subset H$, the distance in $\mathcal{AC}(F)$ from the projection of $D$ into $F$ to any arc in the projection of $D'$ is at most two. 

Isotope $D'$ so that $D \cap D'$ is a collection of arcs, properly embedded in each of the disks. Let $E \subset D'$ be an outermost disk bounded by an arc of intersection. Assume for contradiction that $\partial E$ is either disjoint from $F'$ or intersects $F'$ in a single arc. (This arc may be disjoint from $D$ or may have one or both endpoints in $\partial D$.) Let $\alpha = E \cap D$ be the arc of $\partial E$ that is not contained in $\partial H$. The arc $\alpha$ cuts $D$ into two disk, which we will call $D_0$ and $D_1$. 

If either of $D_0$ or $D_1$ intersects $F'$ in strictly more arcs than $E$ then replacing $D$ with the union of $E$ and the other disk will produce a new compressing disk for $H$ whose boundary intersects $F'$ in strictly fewer arcs than $D$. These arcs, indicated by the thick lines in the left and center pictures in Figure~\ref{fig:EandD}, may be in the interior of $E \cap \partial H$ or may contain endpoints of $E \cap \partial H$. By assumption, $D \cap F'$ was minimal, so this is impossible. However, if $D_0 \cap F'$ and $D_1 \cap F'$ each consists of at most one arc then $D \cap F$ will consist of at most two arcs. This contradiction one of our initial assumptions, so we conclude that $E \cap F'$ must contain at least two arcs.
\begin{figure}[htb]
  \begin{center}
  \includegraphics[width=4.5in]{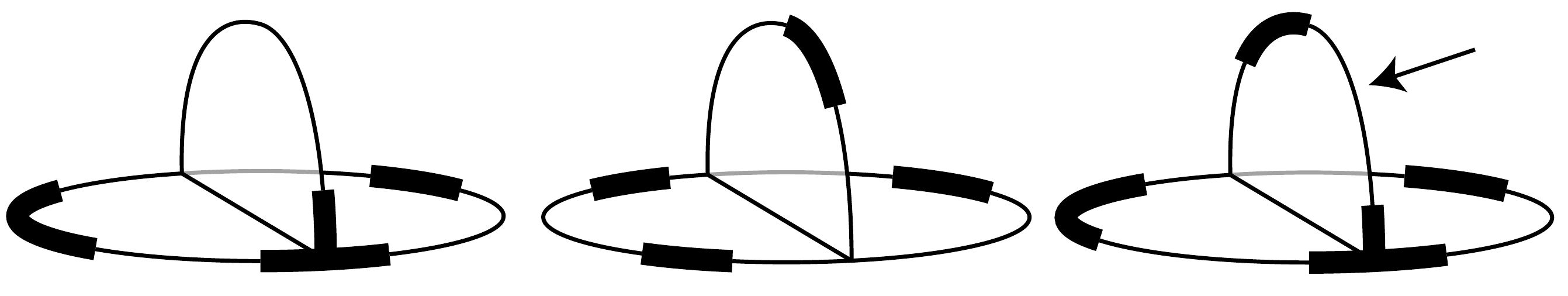}
  \put(-280,32){$E$}
  \put(-300,12){$D$}
  \caption{The compressing disk $D$ and the outermost disk $E \subset D'$. Thick lines indicate where the disks intersect $F'$ and the arrow points to an arc of $\partial E$ that is properly embedded in $F$.}
  \label{fig:EandD}
  \end{center}
\end{figure}

Because there are two arcs of $E \cap F'$, there is also an arc of $E \cap F$ in the interior of the arc $E \cap \partial H$. This subarc, indicated by the arrow on the right side of Figure~\ref{fig:EandD}, is properly embedded in $F$. By construction, the interior of $E \cap \partial H$ is disjoint from $\partial D$ so this arc is in the projection of $\partial D'$ into $F$ and disjoint from the projection of $\partial D$. Any two arcs in the projection of $D'$ are disjoint, so the distance from any arc in the projection of $D$ to any arc in the projection of $D'$ is at most two. Because $D'$ was aritrary, if $D''$ is another compressing disk for $H$ then we can apply the same argument to conclude that the projection of $\partial D''$ to $F$ is distance at most two from any arc of the projection of $D$. By applying the triangle inequality, we see that the distance in $\mathcal{AC}(F)$ from any arc of the projection of $D'$ to any arc of the projection of $D''$ is at most four.
\end{proof}

\begin{Lem}
\label{lem:punctureinclusion}
Let $\Sigma$ be a compact, connected, closed, orientable surface of genus at least two and let $F \subset \Sigma$ be the complement of an open disk (a puncture) in $\Sigma$. Let $\alpha$ and $\beta$ be essential loops in $F$. Then $d_F(\alpha,\beta) \geq \frac{1}{2} d_\Sigma(\alpha, \beta)$.
\end{Lem}

\begin{proof}
Let $\alpha_0,\alpha_1,\ldots,\alpha_n$ be a path in $\mathcal{AC}(F)$ such that $\alpha_0 = \alpha$ and $\alpha_n = \beta$. For each $i$, if $\alpha_i$ is a simple closed curve then it determines a vertex $v_i$ in $\mathcal{C}(\Sigma)$. If $\alpha_i$ is a properly embedded arc then we can extend it across the disk $\Sigma \setminus F$ to form an essential loop, which also determines a vertex $v_i$ in $\mathcal{C}(\Sigma)$. If two curves $\alpha_i$, $\alpha_{i+1}$ are arcs in $F$ then the loops that they define intersect in at most one point, so the vertices $v_i$, $v_{i+1}$ are distance at most two in $\mathcal{C}(\Sigma)$. If either of $\alpha_i$, $\alpha_{i_1}$ is a loop then the corresponding loops in $\Sigma$ are disjoint. Thus the distance from $v_0$ to $v_n$ is at most $2n$ in $\mathcal{C}(\Sigma)$, so $d_F(\alpha,\beta) \geq \frac{1}{2} d_\Sigma(\alpha, \beta)$.
\end{proof}

Note that this Lemma is not true if $F$ is the complement in $\Sigma$ of two or more open disks (punctures). In particular, if we choose $\alpha$ and $\beta$ so that two distinct punctures are contained in the same component of $\Sigma \setminus (\alpha \cup \beta)$ then there will be an arc between these punctures in the complement of $\alpha$ and $\beta$. Thus $d_F(\alpha, \beta)$ will be two, regardless of $d_\Sigma(\alpha,\beta)$.

\section{Constructing the first Heegaard splitting}
\label{sect:splitting1}

We will construct a three-manifold with boundary based on the flexipath $\ell_0,\ldots,\ell_d$ as follows: Consider the product $\Sigma \times [0,1]$. Let $N_0 \subset \Sigma \times \{0\}$ be a regular neighborhood of $\ell_0 \times \{0\}$ and let $N_1 \subset \Sigma \times \{1\}$ be a regular neighborhood of $\ell_d \times \{1\}$. Let $M'$ be the result of gluing one two-handle to $\Sigma \times [0,1]$ along $N_0$ and a second two-handle along $N_1$. In other words, we glue a copy of $D^2 \times [0,1]$ along each of $N_0$, $N_1$ so that the annulus $\partial D^2 \times [0,1]$ is identified to the annulus $N_0$, $N_1$, respectively.

\begin{Def}
A three-manifold $M'$ produced by gluing two-handles to $\Sigma \times [0,1]$ along $\ell_0 \times \{0\}$ and $\ell_d \times \{1\}$, where $\ell_0, \ell_d \subset \Sigma$ are essential, separating loops will be called the \textit{path manifold} defined by $\ell_0, \ell_d$.
\end{Def}

While a path manifold is determined by only two loops $\ell_0,\ell_d$, in practice these loops will be the endpoints of a path, hence the name path manifold. Because the loops are separating in $\Sigma$, by assumption, a path manifold $M'$ will have four boundary components: two that intersect $\Sigma \times \{0\}$ and two that intersect $\Sigma \times [0,1]$. The surface $\Sigma \times \{\frac{1}{2}\} \subset \Sigma \times [0,1]$ is a Heegaard surface for $M'$ dividing it into two compression bodies. In particular, $\Sigma \times \{\frac{1}{2}\}$ defines a partition of the boundary components of $M'$ that divides the two boundary components that intersect $\Sigma \times \{0\}$ from the two components that intersect $\Sigma \times \{1\}$. 
\begin{figure}[htb]
  \begin{center}
  \includegraphics[width=4.5in]{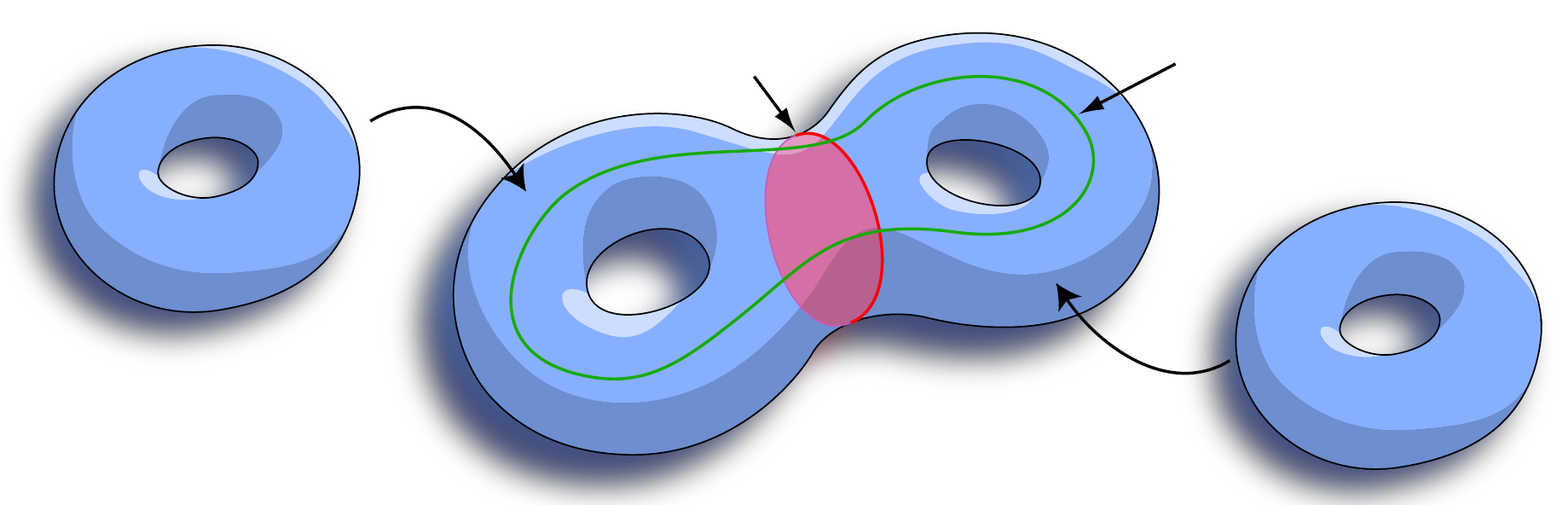}
  \put(-175,92){$\ell_0$}
  \put(-165,10){$\Sigma \times \{0\}$}
  \put(-290,30){$X$}
  \put(-40,65){$X$}
  \put(-80,92){$\ell_d$}
  \caption{Filling two boundary components of a path manifold $M'$.}
  \label{fig:filling}
  \end{center}
\end{figure}
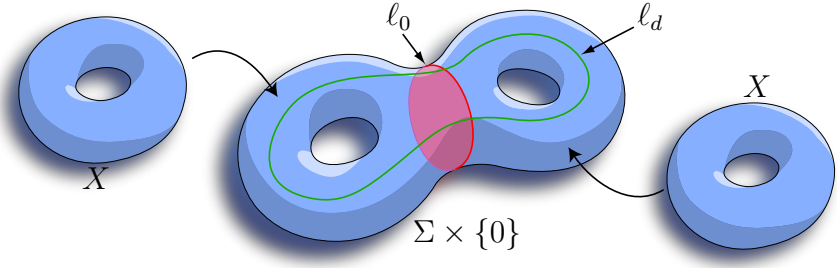

For any other partition of the boundary components of $M'$, there will also be a Heegaard surface and these surfaces will be in distinct isotopy classes from $\Sigma$ because isotopies preserve the partitions of the boundary. However, stabilizing a Heegaard splitting also preserves the boundary partitions, so when considering different Heegaard surfaces in a given manifold, one usually assumes that they induce the same boundary partition. We will construct $M$ by filling in all the boundary components of $M'$, so that the only partition of $\partial M = \emptyset$ is the empty partition. The incompressible surface promised in Theorem~\ref{thm:main1} and the alternate Heegaard surface promised by Theorem~\ref{thm:main2} will be induced by surfaces that come from different boundary partitions of $M'$. 

We will say that $M = X \cup M'$ is a \textit{handlebody filling} along $T \subset \partial M'$ of a three-manifold $M'$ with boundary if $X$ is a handlebody and $M$ is the result of gluing $\partial X$ to $T$ by a homeomorphism.

\begin{Def}
Assume $T$ is a boundary component of a path manifold $M'$ such that $T \cap (\Sigma \times \{0\})$ is a non-empty subsurface $F$ and let $M'' = X \cup M'$ be a handlebody filling along $T$. We will say that $M''$ is a \textit{distance-$k$ filling} of $M'$ with respect to a set of loops $\{\ell_0,\ldots,\ell_d\}$ if for every essential disk $D \subset X$ and each $i$ such that $\ell_i \cap F$ is non-empty, we have $d_F(\partial D, \ell_i) > k + 4$. We require a symmetric condition in the case when $T \cap (\Sigma \times \{1\})$ is non-empty.
\end{Def}

Bounding the distance by $k + 4$ rather than just $k$ may seem odd, but it will be clear later why this notation is more efficient. If $M''$ is a distance-$k$ filling of $M'$, then we can further glue handlebodies into the remaining three boundary components. If the same condition applies to each of these fillings then we will call the resulting three-manifold $M$ a \textit{closed distance-$k$ filling} of $M'$.

These distance-$k$ fillings are very similar to a method used by Qiu-Zou-Gou~\cite{QiuZouGuo} to find handlebody fillings that preserve the distance of a Heegaard splitting. It is not immediately obvious that we can find a distance-$k$ filling for $M'$, since a handlebody of genus greater than one has infinitely many essential disks, defining an infinite diameter set in the curve complex of its boundary. However, it turns out that it is possible.

\begin{Lem}
Let $M'$ be a path manifold defined by separating loops $\ell_0 \times \{0\}$ and $\ell_d \times \{1\}$ such that $d(\ell_0, \ell_d) > 1$. Then for any positive integer $k$, and every finite set of loops $\ell_0,\ldots,\ell_d$, there is a distance-$k$ filling $M$ of $M'$.
\end{Lem}

\begin{proof}
Assume, without loss of generality, that $F$ is a subsurface of $\Sigma \times \{0\}$ rather than $\Sigma \times \{1\}$. First note that since $d(\ell_0, \ell_d) > 1$, the intersection $\ell_0 \cap \ell_d$ is non-empty, so the projection of $\ell_d$ into $F$ is well defined. We will consider two cases, based on whether the genus of the boundary component $T$ has genus one or genus strictly greater than one. 

First assume the genus of $T$ is one, so $X$ is a genus-one handlebody. Let $\mu$ be the loop bounding the unique (up to isotopy) essential disk $D \subset X$. The subsurface $F$ will be a once-puntured torus and $\mathcal{AC}(F)$ has infinite diameter, so there is an essential, non-separating loop $\mu'$ in $F$ whose distance from the projection into $F$ of $\ell_d$ is at least $k + 4$. If we glue $\partial X$ to $T$ by a map that sends $\mu$ onto $\mu'$, we will guarantee that $d_{F}(\partial D, \ell_d) \geq k + 4$. 

Otherwise, if $T$ has genus greater than one, the set of loops bounding essential disks in $X$ will span a subcomplex of $\mathcal{C}(T)$ called the \textit{disk set} or \textit{handlebody set} which has infinite diameter in $\mathcal{C}(\Sigma)$. However, one can find loops in the curve complex that are arbitrarily far from the handlebody set, for example by using the techniques in Hempel's paper~\cite{Hempel}.   In fact, Hempel's construction of high distance Heegaard splittings is much stronger than this statement. 

Thus, since $\{\ell_i\}$ is a finite set, we can glue $X$ into $T$ so that the boundary of each compressing disk for $D$ is distance at least $2k + 8$ from any loop in $T$ formed by an arc of intersection between a loop $\ell_i$ and $F$. Then by Lemma~\ref{lem:punctureinclusion}, the distance in $\mathcal{AC}(F)$ from the projection of any $\ell_i$ to the projection of any $\partial D$ is at least $k + 4$, i.e.\ $d_{F}(\partial D, \ell_i) \geq k + 4$.
\end{proof}

The following Lemma shows that the distance-$k$ fillings have exactly the property we will want:

\begin{Lem}
Let $M'$ be a path manifold defined by separating loops $\ell_0$ and $\ell_d$ such that $\ell_0 \cap \ell_d$ contains strictly more than four points. Let $M$ be a closed, distance-$k$ filling of $M'$ with respect to $\ell_0$, $\ell_d$ for $k > d$. Then the Heegaard surface $\Sigma \subset M$ defined by $\Sigma \times [0,1]$ has distance $d(\Sigma) = d(\ell_0, \ell_d)$. 

Moreover, for each subsurface $F \subset \Sigma$ in the complement of $\ell_0$ or in the complement of $\ell_d$, every path in $\mathcal{C}(\Sigma)$ between disks on opposite sides of $\Sigma$ contains a vertex representing a loop that can be isotoped disjoint from $F$.
\end{Lem}

Note that the Lemma does not require that the fillings be distance $k$ with respect to all the loops in the path from $\ell_0$ to $\ell_d$, only to the end vertices. We will use the more general condition later on.

\begin{proof}
We will prove the second half of the statement first. Let $D^-$ be an essential disk on the side of $\Sigma \subset M$ containing $\Sigma \times \{0\}$ and let $D^+$ be an essential disk on the other side of $\Sigma$. By construction, the loop $\ell_0$ bounds a disk $D_0$ on the same side of $\Sigma$ as $D^-$ and the loop $\ell_d$ bounds a disk $D_d$ on the same side as $D^+$.

Assume $\partial D^-$ has been isotoped to intersect $\ell_d$ and $\ell_0 = \partial D_0$ minimally, then made transverse to $D_0$ so that $D^- \cap D_0$ is a (possibly empty) collection of arcs properly embedded in the two disks. 

If $D_0 \cap D^-$ is non-empty, let $E^- \subset D^-$ be an outermost disk cut off by an arc $\alpha$ of $D^- \cap D_0$ and let $E_0 \subset D_0$ be one of the two components of $D_0 \setminus \alpha$. Let $D'_0$ be the result of isotoping the union $E^- \cup E_0$ disjoint from $E^-$ and $D_0$. Note that $D'_0$ will, in general, intersect $D^-$ non-trivially, but it is disjoint from $D_0$. If we let $\beta$ be the arc $E \cap \Sigma$ and let $F$ be the component of $\Sigma \setminus \ell_0$ that contains $\beta$ then $\partial D_0$ will be contained in $F$ and disjoint from $\beta$. Thus $d_F(\partial D^-, \partial D'_0) = 1$.

On the other hand, if the intersection $D_0 \cap D^-$ is empty then either $\partial D^-$ is an essential loop in a component $F$ of $\Sigma \setminus \ell_0$ or $D^-$ is parallel to $D_0$. In either case, define $D'_0 = D^-$ and we find that $d_F(\partial D^-, \partial D'_0) = 0$.

If $D'_0$ is parallel to $D_0$ (which is the case when $D^-$ is parallel to $D_0$) then by assumption, $D'_0 \cap \ell_d = D_0 \cap \ell_d > 4$. Otherwise, because $M$ is a distance $k$ filling of $M'$, we have $d_F(\partial D'_0, \ell_d) > k + 4$. Since $\partial D'_0$ is disjoint from a component of $\partial D^- \cap F$ and any two components of $\partial D^- \cap F$ are disjoint, the distance from any arc of $\partial D^- \cap F$ to $\partial D'_0 \cap F$ is at most two. By the triangle inequality, this implies $d_F(\partial D^-, \ell_d) > k + 2$. 

In particular, since $d_F(\partial D^-, \ell_d) > 2$, the intersection $\partial D^- \cap \ell_d$ must contain strictly more than four points. Since $D^-$ was arbitrary, we conclude that any disk $D^-$ intersects either subsurface $F'$ in the complement of $\ell_k$ in strictly more than two arcs, so Lemma~\ref{lem:diskprojection} implies that the projection from the disk set on the negative side of $\Sigma$ into either component of $\Sigma \setminus \ell_d$ has diameter at most four. By a symmetric argument, we conclude that the projection from the disk set on the positive side of $\Sigma$ into either component of $\Sigma \setminus \ell_0$ has diameter at most four.

By assumption, $d_F(\partial D^-, \ell_k) > k + 4$ for any disk $D^-$ so by the triangle inequality we have $d_F(\partial D^-, \partial D^+) > k$ for each subsurface $F$ in the complement of $\ell_0$ or $\ell_k$ that intersects both $\partial D^-$ and $\partial D^+$ non-trivially. Thus if $\partial D^- = \ell'_0, \ell'_1,\ldots,\ell'_n = \partial D^+$ is a path in $\mathcal{C}(\Sigma)$ of length $n < k$ then Lemma~\ref{lem:mustmissF} implies that some loop in this path (possibly $\partial D^-$ or $\partial D^+$) is disjoint from $F$. This confirms the second half of the statement of the Lemma.

For the first half, let $\ell'_i, \ell'_j$ be the loops of the path that are disjoint from the two components of $\Sigma \setminus \ell_0$. It is possible that $i = j$, in which case we must have that $\ell'_i = \ell'_j = \ell_0$, since this is the only essential loop that can be isotoped disjoint from both of these subsurfaces. Otherwise, without loss of generality assume $i < j$. Since $\ell'_j$ is disjoint from a subsurface bounded by $\ell_0$, it is also disjoint from $\ell_0$. Since $\ell_0$ bounds the disk $D_0$, and $i > 0$, we can find a new path $\partial D_0 = \ell_0, \ell'_i,\ell'_{i+1},\ldots,\ell'_n = \partial D^+$ whose length is less than or equal to that of the original path.

A similar argument at the other end of the path implies that there is a path from $\partial D_0$ to $\partial D_d$ whose length less than or equal to the length $n$ of the original path. Thus we must have $n \geq d$. Since $D^-$ and $D^+$ were arbitrary disks on opposite sides of $\Sigma$, this implies that $d(\Sigma) = d(\partial D_0, \partial D_d) = d$.
\end{proof}

\begin{Coro}
\label{coro:filledist}
If $M$ is a closed, distance-$k$ filling with respect to $\ell_0$, $\ell_d$ of a path manifold $M'$ defined by the endpoints $\ell_0$, $\ell_d$ of a $(d,k)$ flexipath $\ell_0,\ldots,\ell_d$ such that $d$ is even and $k > d \geq 4$ then the Heegaard surface $\Sigma$ for $M$ induced by this construction has distance $d(\Sigma) = d$.
\end{Coro}

\section{Flat surfaces}
\label{sect:flat}

Let $M$ be a three-manifold that results from a distance $k$ filling of a path manifold $M'$ defined by a $(d,k)$ flexipath $\ell_0,\ldots,\ell_d$. Recall that $M$ is the union of $\Sigma \times [0,1]$ with a two-handle and two handlebodies glued onto each boundary surface $\Sigma \times \{0\}$, $\Sigma \times \{1\}$. One of the two-handles is attached along the loop $\ell_0 \times \{0\}$ and thus contains a disk $V_0$ whose boundary is $\ell_0 \times \{0\}$. Similarly, there is a disk $V_d$ with boundary $\ell_d \times \{1\}$ in the two-handle attached to $\Sigma \times \{1\}$.
\begin{wrapfigure}{R}{.35\textwidth}
  \label{fig:verticalabels}
  \ffigbox[\FBwidth]{\caption{}}
    {\includegraphics[width=\textwidth]{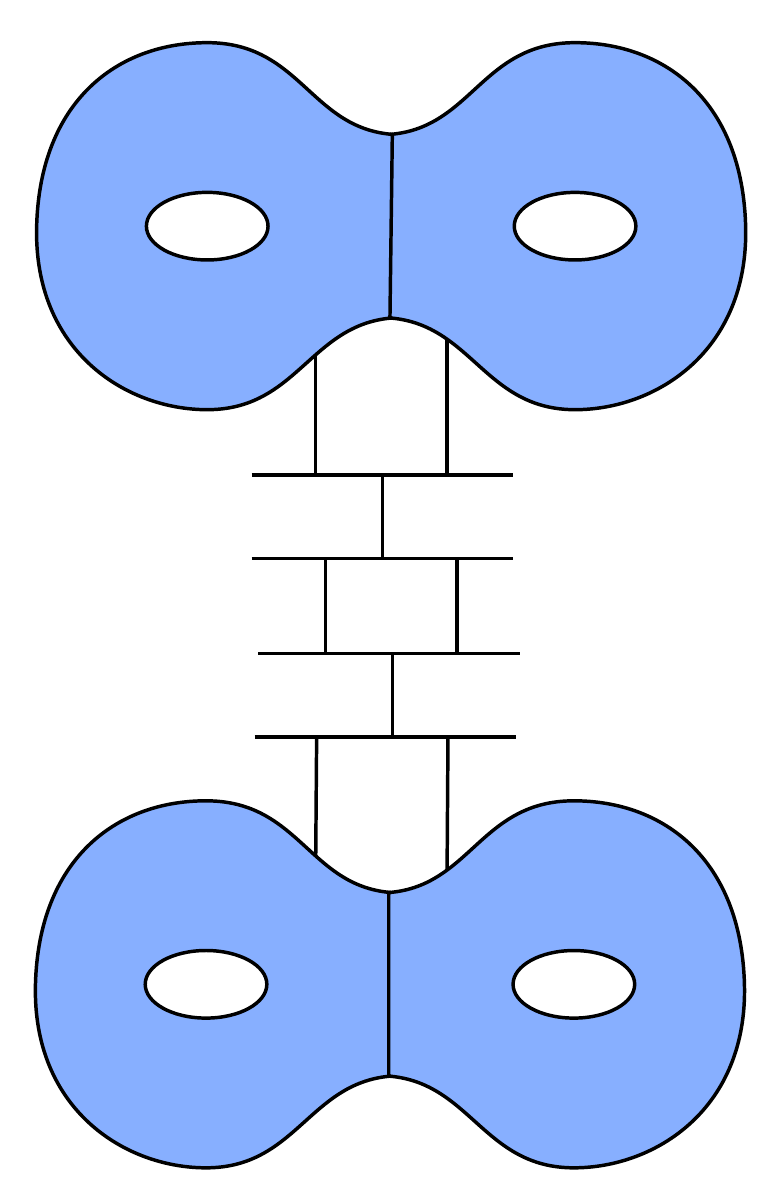}
    \put(-60,30){$V_0$}
    \put(-70,60){$V_1$}
    \put(-60,77){$V_2$}
    \put(-68,92){$V_3$}
    \put(-60,107){$V_4$}
    \put(-68,126){$V_5$}
    \put(-60,150){$V_6$}
    \put(-93,71){$t_1$}
    \put(-93,86){$t_2$}
    \put(-93,102){$t_3$}
    \put(-93,116){$t_4$} }
  \vspace{-15pt}
\end{wrapfigure}

The union of the two-handle attached along $\ell_0 \times \{0\}$ and the handlebodies used to fill the adjacent boundary components of $M'$ form a handlebody $H^- \subset M$. The second two-handle and the remaining two filling handlebodies form a second handlebody $H^+ \subset M$ such that $M = H^- \cup (\Sigma \times [0,1] \times H^+$. In particular, the disk $V_0$ is properley embedded in $H^-$, while the disk $V_d$ is properly embedded in $H^+$.

We will construct an incompressible surface in $M$ by connecting $V_0$ to $V_d$ by a collection of vertical and horizontal surfaces in $\Sigma \times [0,1]$ defined by the flexipath. For each $i < d$, define $t_i = \frac{i}{d-1}$ so that $t_0,t_1,\ldots,t_{d-1}$ is an increasing sequence of points with $t_0 = 0$ and $t_{d-1} = 1$. By construction, each $\ell_i$ with $i$ even is a separating loop in $\Sigma$. For these values of $i$ strictly between $0$ and $d$, define $V_i = \ell_i \times [t_{i-1}, t_i]$.

For each odd value of $i$, let $N \subset \Sigma$ be the closure of a regular neighborhood of $\ell_i$ and let $L_i$ be the union of the two boundary loops of $N$. For these values of $i$, define $V_i = L_i \times [t_{i-1}, t_i]$. To keep the notation consistent, we will define $L_i = \ell_i$ when $i$ is even. Thus each $V_i$ is either a disk (for $i = 0$ and $i = d$), an annulus (for other even values of $i$) or a pair of annuli (for odd $i$). These will be the vertical parts of the closed surface we are about to construct. They're shown schematically in Figure~\ref{fig:verticalabels}.

For each even value of $i < d$, the surface $\Sigma \times \{t_i\}$ will intersect $V_i$ in the loop $\ell_i \times \{t_i\}$ and will intersect $V_{i+1}$ in two loops parallel to $\ell_{i+1} \times \{t_{i+1}\}$. By construction, $\ell_i$ bounds a once-punctured torus containing $\ell_{i+1}$ so these three loops in $\Sigma \times \{t_i\}$ cobound a pair-of-pants which we will call $P_i \subset \Sigma \times \{t_i\}$. This is shown in the top right of Figure~\ref{fig:flatsurface}. For odd values of $i$, the situation is similar, with a pair-of-pants $P_i \subset \Sigma \times \{t_i\}$ cobounded by $\ell_{i+1}$ and the two loops parallel to $\ell_i$.
\begin{figure}[htb]
  \begin{center}
  \includegraphics[width=4.5in]{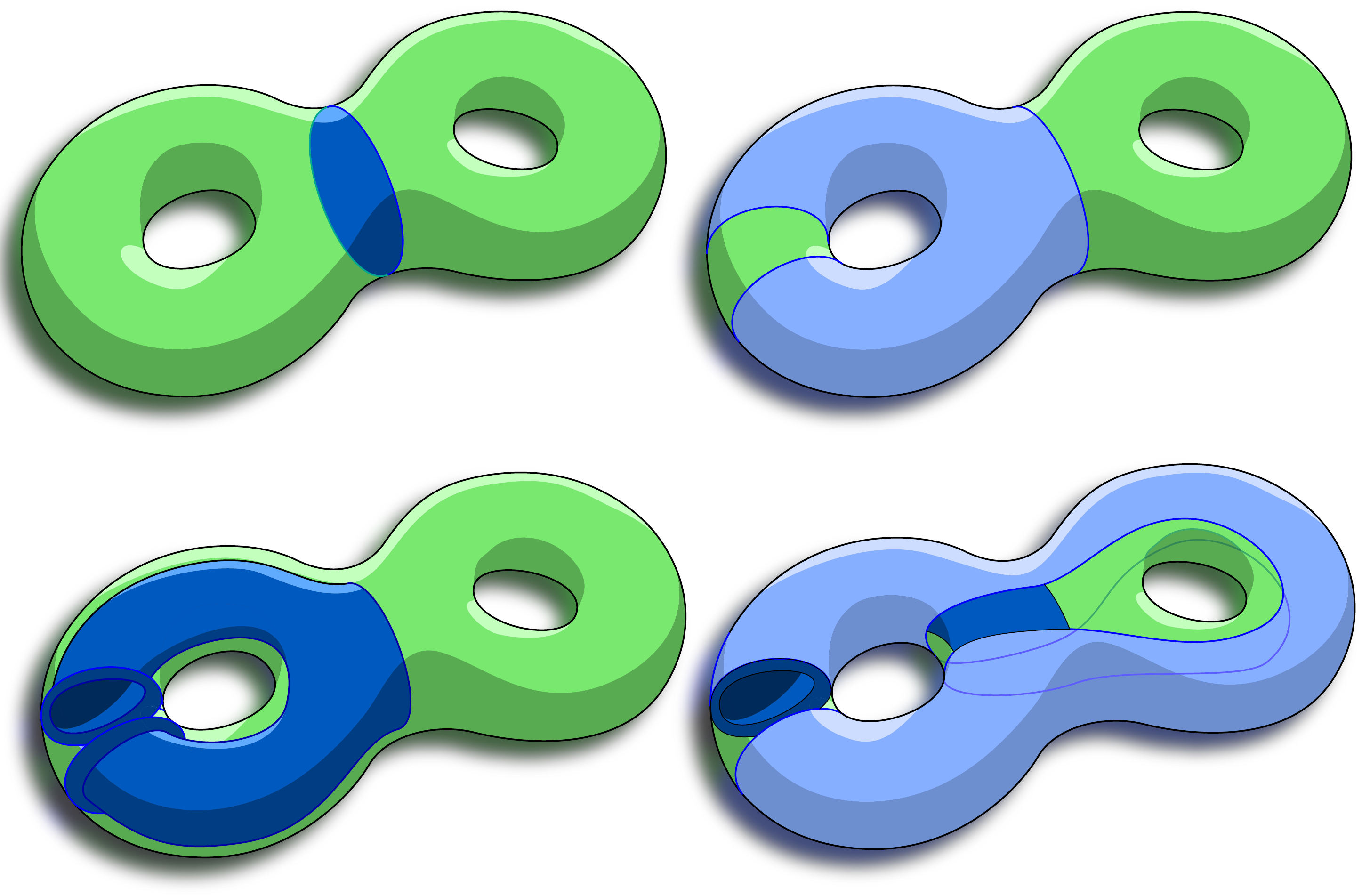}
  \put(-240,130){$V_0$}
  \put(-95,115){$P_0$}
  \put(-245,10){$S_1$}
  \put(-80,20){$P_1$}
  \caption{The intermediate surfaces leading up to $S$.}
  \label{fig:flatsurface}
  \end{center}
\end{figure}
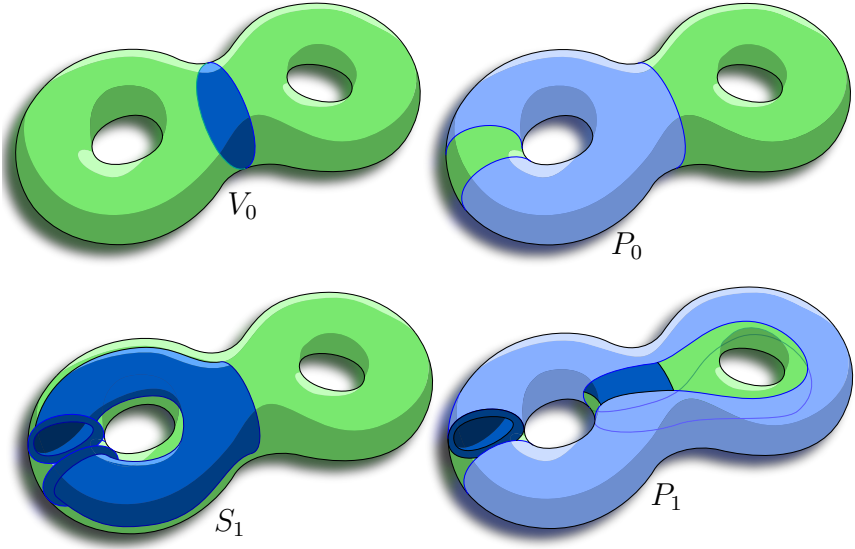

\begin{Def}
The union $S$ of all the surface $\{V_i\}$, $\{P_i\}$ will be called the \textit{flat surface} defined by the flexipath $\ell_0,\ldots,\ell_d$.
\end{Def}

The flat surface is shown schematically in Figure~\ref{fig:frombothsides}, above the flexipath $\ell_0,\ldots,\ell_d$ used to define it. Notice that each pair-of-pants corresponds to an edge in the curve complex.
\begin{figure}[htb]
  \begin{center}
  \includegraphics[width=4.5in]{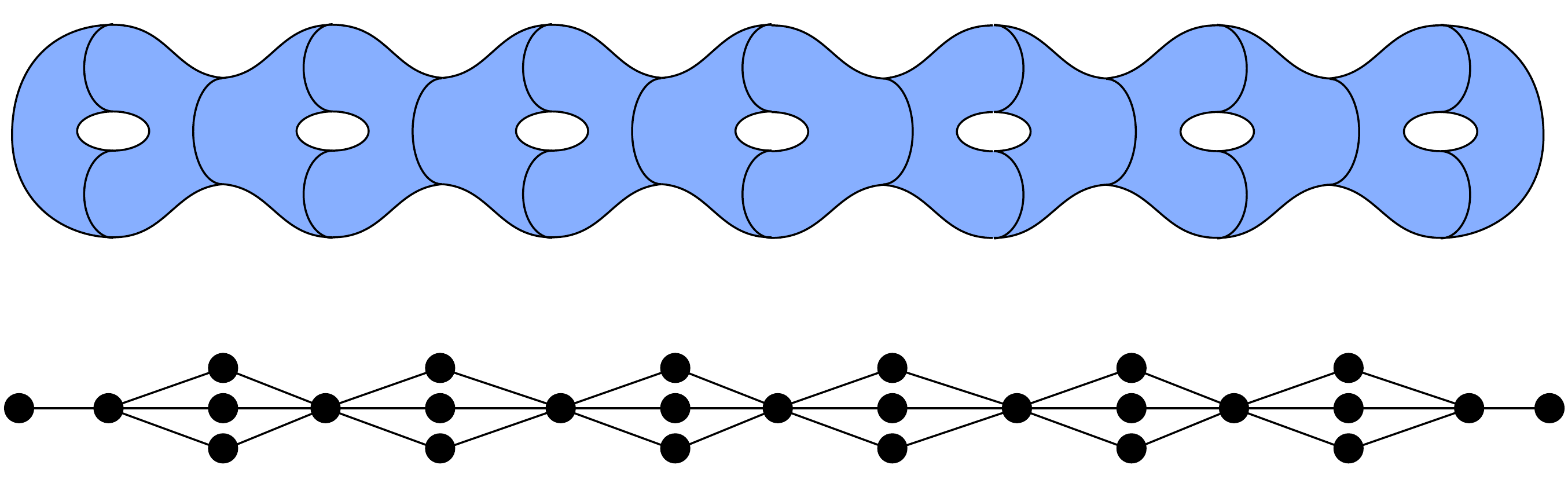}
  \caption{The flat surface $F$ follows the path used to define it.}
  \label{fig:frombothsides}
  \end{center}
\end{figure}

In order to understand the flat surface $S$, we will look at the intermediate surfaces $S_i = V_0 \cup P_0 \cup V_1 \cup P_1 \cup \cdots \cup V_i$. Note that for $i < d$, this surface always stops at a vertical annulus or annuli, as in the bottom left of Figure~\ref{fig:flatsurface}, so $S_i$ will be properly embedded in a handlebody $H_i = H^- \cup (\Sigma \times [0,t_i])$ bounded by $\Sigma \times \{t_i\}$.

We will begin by considering the complement in $H_i$ of $S_i$, each component of which is a union of products of the form $F \times [t_j, t_{j+1}]$, glued along subsurfaces of their horizontal boundaries.

\begin{Lem}
\label{lem:primitivegluing}
Let $X = X_0 \cup X_1$ be the union of two handlebodies $X_0$, $X_1$ glued along subsurfaces $F_0 \subset \partial X_0$ and $F_1 \subset \partial X_1$. Assume there is a collection of essential disks $E_0,\dots E_k \subset X_0$, each of which intersects $F_0$ in a single arc such that the complement in $F_0$ of a regular neighborhood of these arcs is a collection of one or more disks. Then $X$ is a handlebody.
\end{Lem}

This is a generalization of a well known result that gluing two handlebodies along annuli such that one of the annuli is primitive in its respective handlebody, produces a new handlebody. In particular, this special case occurs when $F_0$ is an annulus and there is a single disk $E_0$.

\begin{proof}
Let $X' \subset X$ be the closure of a regular neighborhood in $X_0$ of $F_0 \cup E_0 \cup \ldots \cup E_k$. Then the intersection of $X'$ with the closure of its complement $X_0 \setminus X'$ is a surface $D'$ that results from isotoping the interior of $F_0$ into the interior of $X_0$, then boundary compressing the resulting surface $F'$ along the disks $E_0,\ldots,E_k$. In other words, each disk $E_0$ intersects $F'$ in a properly embedded arc and we construct $D'$ by replacing a regular neighborhood in $F'$ of each of these arcs with two disks parallel to $E_i$, as in Figure~\ref{fig:gluehbodies}.
\begin{figure}[htb]
  \begin{center}
  \includegraphics[width=3.5in]{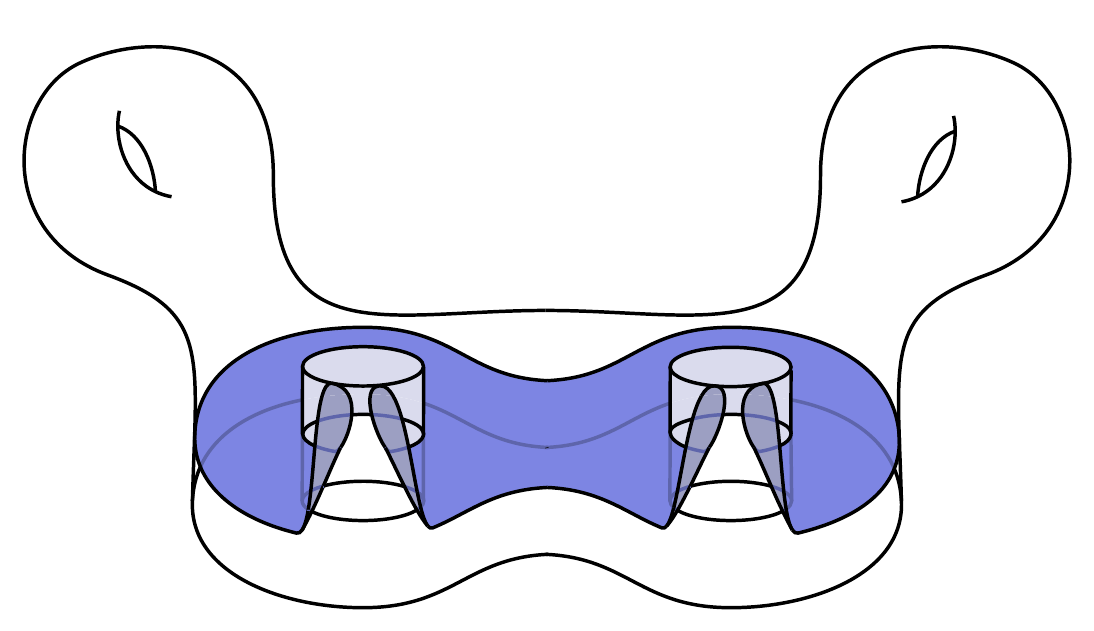}
  \caption{The surface $D$ defined by a regular neighborhood of $F_0 \cup E_0 \cup \ldots \cup E_k$.}
  \label{fig:gluehbodies}
  \end{center}
\end{figure}

By construction, the surface $D$ is properly embedded in $X$ and homeomorphic to the complement in $F_0$ of a regular neighborhood of the arcs $E_i \cap F_0$. This complement is a collection of disks by assumption.

We can form $X$ by gluing $X'$ to $X_1$ along $F_0$, then gluing the closure of $X_0 \setminus X'$ to this union along $D$. By construction, $X_0$ is homeomorphic to $F_0 \times [0,1]$ by a map that sends $F_0$ to $F_0 \times \{0\}$. Thus if we glue $X'$ to $X_1$ along $F_0$, the resulting space will be homeomorphic to $X_1$, i.e.\ a handlebody. If we then glue $X_0 \setminus X'$ to this handlebody along the collection of disks $D$, the resulting space will be a new handlebody homeomorphic to $X$.
\end{proof}

\begin{Coro}
\label{lem:newsurfacesincomp}
For $i \leq d$, the closure of each component of $H_i \setminus S_i$ is a handlebody.
\end{Coro}

\begin{proof}
We will prove this by induction on $i$. For the base case $i = 0$, $S_0$ is a separating disk in $H_0 = H^-$ that cuts $H_0$ into two handlebodies. For the inductive step $i > 0$, each component of $H_i \setminus S_i$ is the union of a component of $H_{i-1} \setminus S_{i-1}$ with the set $F \times [t_{i-1}, t_i]$ for $F$ a component of $\Sigma \setminus L_i$. These two sets intersect along the subsurface $F' \times \{t_{i-1}\}$ of $\Sigma \times \{t_{i-1}\}$ such that $F'$ is either equal to $F$ or a subsurface of $F$ bounded by $L_{i-1}$, depending on which component of $\Sigma \setminus \ell_{i+1}$ is the subsurface $F$. 

In either case, there is a collection of essential arcs $\alpha_i$ in $F$ such that each $\alpha_i$ intersects $F'$ in a single subarc (possibly all of $\alpha_i$) and the complement in $F'$ of these arcs is a single disk. Each $\alpha_i$ determines an essential disk $\alpha_i \times [t_{i-1}, t_i]$ in the handlebody $F \times [t_{i-1},t_i]$ whose boundary intersects $F'$ in a single arc so Lemma~\ref{lem:primitivegluing} implies that the union is a handlebody. By induction, this proves the Lemma for every value $i < d$. 
\end{proof}

Note that while the surface $S_d$ is also defined, the proof of Corollary~\ref{lem:newsurfacesincomp} does not apply to it because the complementary components of $S_d$ do not result from gluing a product to the complementary components of $S_{d-1}$. Instead, they result from gluing handlebodies to these components along arbitrary subsurfaces of their boundaries. As we will see below, the surface $S_{d+1}$ will often be incompressible.

\section{Incompressible surfaces}
\label{sect:incompressible}

At this point, it will be helpful to introduce some more terminology related to the surface $S_i$. As we saw in the proof of Corollary~\ref{lem:newsurfacesincomp}, each component $X$ of $H_i \setminus S_i$ is the union of a component $X'$ of $H_{i-1} \setminus S_{i-1}$ with a product $F \times [t_{i-1},t_i]$ along a subsurface $F' \times \{t_{i-1}\}$. We will call each $F \times [t_{i-1},t_i]$ a \textit{block at level $i$} and we will call $F' \times \{t_{i-1}\}$ the \textit{connecting surface}. The surface $F$ will be called the \textit{block surface}.
\begin{wrapfigure}{R}{.35\textwidth}
  \ffigbox[\FBwidth]{\caption{}}
    {\includegraphics[width=\textwidth]{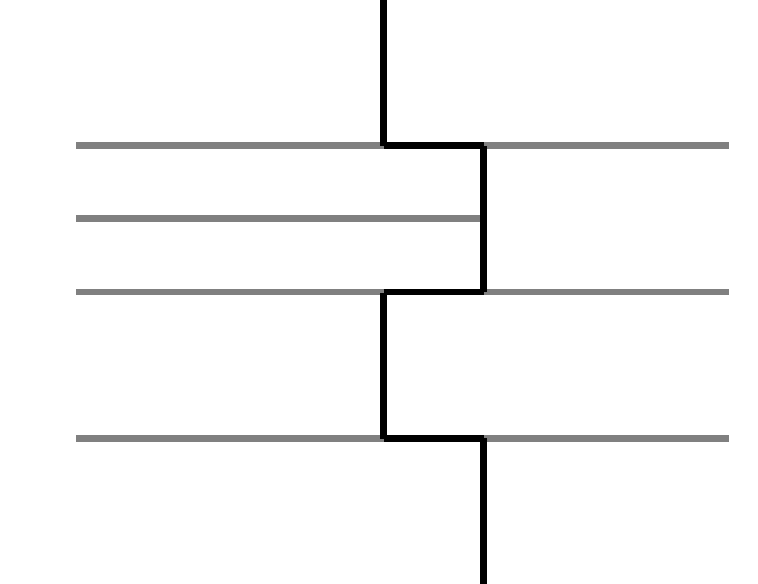}
    \put(-125,45){$F'$}
    \put(-125,57){$F$}
    \put(-105,32){outside}
    \put(-50, 32){inside}}
  \label{fig:inoutside1}
  \vspace{-5pt}
\end{wrapfigure}

The block surface is bounded by $L_i$, which is the loop $\ell_i$ when $i$ is even and two parallel copies $\ell_i$ when $i$ is odd. If $L_{i-1}$ is contained in the block surface $F$ then it is the boundary of the connecting subsurface $F'$, which will be a proper subsurface of $F$. In this case, $L_i \times \{t_i\}$ will be the boundary of the connecting surface in $F \times \{t_i\}$ for a block at level $i+1$ and we will say that $X$ is an \textit{inside block}. 

Otherwise, if $L_{i-1}$ is not contained in $F$ then the connecting surface $F' \times \{t_{i-1}\}$ will be the entire surface $F \times \{t_{i-1}\}$. Moreover, $L_{i+1}$ will be disjoint from $F$, so that $F \times \{t_i\}$ is the connecting surface at level $i+1$. In this case, we will say that $X$ is an \textit{outside block}. Note that each subset $\Sigma \times [t_{i-1}, t_i]$ contains exactly one inside block and one outside block. These are shown schematically in Figure~13, where the flat surface $S$ is indicated by the black arcs, and the connecting surfaces and one block surface are indicated in grey.

\begin{Lem}
\label{lem:newsurfincomp1}
Let $Y$ be the component of $H_0 \setminus S_0$ whose boundary contains $\ell_1 \times \{t_0\}$ and assume $\ell_1$ intersects every compressing disk for $Y$ in at least two points. Then $S_1$ is incompressible in $H_1$. Moreover, if $D$ is a compressing disk for a component of $H_1 \setminus S_1$ such that $D \cap \partial H_1$ is a single component then $D$ is on the same side of $S_1$ as the inside block at level 1 and the arc or loop of intersection is distance at most one in $F$ from the projection of a compressing disk for $H_1$.
\end{Lem}

\begin{proof}
Let $D$ be a compressing disk for a component $X$ of $H_1 \setminus S_1$ and assume $D \cap \partial H_1$ is connected. (If this intersection is empty then $D$ will be a compressing disk for $S_1$.) Let $X_1$ be the block at level one on the side of $S_1$ that contains $D$. Any arc of $\partial D$ contained in the disk $S_0$ can be isotoped out of this disk and into $P_0$, so that $\partial D$ is disjoint from the interior of $H_0$. We can further isotope $D$ so that it is transverse to the connecting surface $F'_0$ and so that any loop in $D \cap F'_0$ is essential in $\partial H_0$.

If $X_1$ is the outside block at level 1 then there is a value $\epsilon > 0$ such that $D \cap (\Sigma \times [t_1-\epsilon,t_1])$ is either empty, or isotopic to a vertical annulus $\partial D \times [t_1-\epsilon,t_1]$ or a vertical band $(\partial D \cap \partial H_1) \times [t_1-\epsilon,t_1]$. Moreover, there is an ambient isotopy of $X$, fixing $X \cap \partial H_1$, that takes $\Sigma \times [t_1-\epsilon,t_1]$ onto $X_1$ and sends the rest of $X_1$ into $X_0$. After this $D \cap X_1$ will be empty or a vertical disk or annulus and $D \cap X_0$ will be a single disk whose boundary is either disjoint from $F_1$ or intersects $F_1$ in a single loop or one arc. The boundary of such a disk intersects each component of $L_1$ at most once, contradicting our initial assumption that every disk in $H_0$ intersects $\ell_1$ in at least two points. Thus we conclude that $X_1$ must be an inside block.

Consider the vertical annulus $A = \ell_0 \times [t_0,t_1]$, which is properly embedded in $X_1$, as indicated in Figure~\ref{fig:lem20}. The boundary of $A$ separates $S_0$ from the rest of $S_1$ so if $D$ is disjoint from $A$ then $\partial D$ must be contained in the component of $X_1 \setminus A$ containing $X_0$ and thus disjoint from $S_1$. This implies that $D$ is itself a compressing disk for $H_1$ (but not for $S_1$) so $\partial D$ is distance zero from a compressing disk. Otherwise, assume $\partial D \cap H_1$ is an arc or the empty set. Then $D \cap A$ is a collection of essential loops and arcs that are parallel in $A$ into $A \cap \partial H_1$, since none of the arcs can have endpoints in $A \cap \partial H_1$. 

Assume that we have isotoped $D$ so as to minimize $D \cap A$. If $D \cap A$ contains one or more arcs then an outermost arc determines a disk $E$ such that boundary compressing $D$ across $E$ produces two disks. Because $D \cap \partial H_1$ is a single arc, one of the components $D'$ of the resulting surface is a disk with boundary entirely in $H_1$. By construction, $D'$ is disjoint from $D$. Moreover, if $D'$ is a boundary parallel disk then we can isotope $D$ so as to remove the arc that defined $E$. Thus $D'$ must be an essential disk for $H_1$, so $\alpha$ is distance at most one from the projection into $F_1$ of this disk.

This takes care of the case when $D \cap A$ contains an arc. If $D \cap A$ is a collection of essential loops then we have two more cases to consider: In the first case, if $D \cap \partial H_1$ is an arc then it is disjoint from $\partial A$, which is part of the compressing disk $S_0 \cup A$ for $H_1$, so $\alpha$ is again distance at most one from the projection of a compressing disk. 

In the second case, $D \cap A$ is a collection of essential loops. Assume we have isotoped $D$ so that $D \cap A$ is minimal. In particular, note that we can remove any trivial loops of intersection. Let $\beta$ be an arc in $F_0$ with one endpoint $v_0$ in $\ell_0$, one endpoint $v_1$ in $\ell_1$ and interior disjoint from $\ell_0$ and $\ell_1$. Define $B = \beta \times [t_0, t_1]$ and assume that $D$ has been made transverse to $B$ such that $D \cap B$ is a collection of properly embedded arcs in $B$. The endpoints of these arcs will be contained in $\beta \times \{t_0\}$ and in $v_0 \times [t_0, t_0]$, with exactly one endpoint for each loop of $D \cap A$. Two such arcs are indicated on the left in Figure~\ref{fig:lem20}. An outermost arc with both endpoints in $v_0 \times [t_0, t_1] \subset A$ will define an isotopy of $D$ that turns two loops of $D \cap A$ into a single trivial loop, contradicting minimality of $D \cap A$.
\begin{figure}[htb]
  \begin{center}
  \includegraphics[width=4.5in]{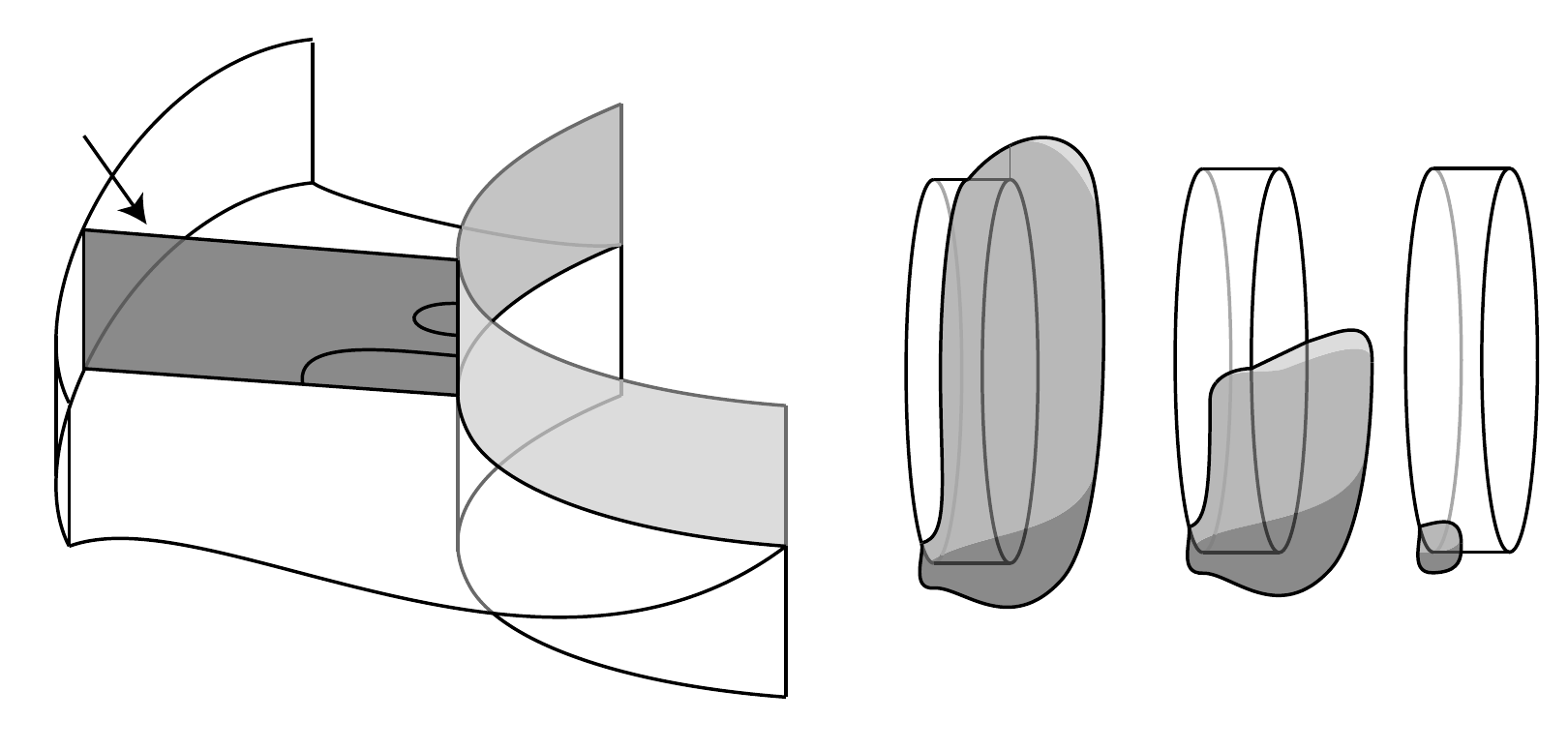}
  \put(-315,130){$B$}
  \put(-215,110){$A$}
  \put(-110,105){$D$}
  \caption{Reducing intersections between the disk $D$ and the annulus $A$.}
  \label{fig:lem20}
  \end{center}
\end{figure}

If the intersection contains an arc from $v_0 \times [t_0, t_1]$ to $\beta \times \{t_1\}$ then this arc defines an isotopy that pulls an arc of $\partial D$ across $A$ and into the disk $S_0$, as on the right in Figure~\ref{fig:lem20}. This isotopy turns a loop of $D \cap A$ into an arc that shares its endpoints with the arc that remains from $\partial D$. If we further isotope the arc $\partial D \cap S_0$ the rest of the way across $S_0$ and back into $P_0$, the resulting disk will have one fewer loops of intersection with $A$, contradicting the assumption that $D' \cap S_1$. This contradiction implies that $S_1$ is incompressible in $H_1$.
\end{proof}

For the next step, we will need the following result by Masur-Schleimer, which is part of Lemma 12.12 in~\cite{Masur-Schleimer}. As usual, $F$ will be a compact, connected, orientable surface with non-empty boundary.

\begin{Lem}[Masur-Schleimer]
\label{lem:masurschleimer}
If $D$ is an essential disk in the handlebody $F \times [0,1]$, isotoped to intersect $\partial F \times [0,1]$ minimally then each arc in the set $\partial D \cap (F \times \{0\})$ will be distance at most 6 from every arc of $\partial D \cap (F \times \{0\})$ in $\mathcal{AC}(F)$.
\end{Lem}

Note that the statement has been translated from the language of holes in the curve complex used in Masur-Schleimer's paper. We will use this Lemma to analyze disks in $X^-_i$, that intersect $\partial H$ in the subsurface $F_i \times \{1\}$ of the boundary of $F_i \times [0,1]$.

\begin{Lem}
\label{lem:newsurfincomp2}
Assume $\ell_0,\ldots,\ell_d$ is a strict $(d,k)$ flexipath for $k > 8$ and $D$ is a compressing disk for either component of $H_i \setminus S_i$ for some $i \leq d$, such that $\partial D$ intersects $\partial H_i$ in one or fewer arcs. Then $X_i$ is an inside block and either $D$ is a vertical disk or $i = 1$ and $\partial D$ is disjoint from a compressing disk for $H_i$. In particular, $S_i$ is incompressible as a properly embedded surface in $H_i$ for $i \geq 1$.
\end{Lem}

Note that each $S_i$ will still be boundary compressible because every properly embedded positive genus surface in a handlebody is either compressible or boundary compressible.

\begin{proof}
We will again prove this by induction on $i$. The base case when $i = 1$ follows from Lemma~\ref{lem:newsurfincomp1}. For the inductive step, we will assume that $S_i$ satisfies the conclusion of the Lemma and let $D$ be a compressing disk for a component $X$ of $H_{i+1} \setminus S_{i+1}$ such that $D \cap \partial H_{i+1}$ has one or zero components. Let $X_1,\ldots,X_{i+1}$ be the blocks in $X$, let $F_1,\ldots,F_{i+1}$ be the block surfaces and let $F'_0,\ldots,F'_i$ be the connecting surfaces.

First, consider the case when $X_{i+1}$ is an inside block. We can isotope $D$ transverse to $F'_i$ and so that $D \cap F'_i$ is a collection of loops and arcs that are essential in $F'_i$. Any component of $\partial D \cap V_{i+1}$ that is disjoint from $\partial H_{i+1}$ is a boundary parallel arc in the vertical annulus $V_{i+1}$ and can be isotopied out of $V_{i+1}$. We can therefore assume that each component of $D \cap X_{i+1}$ intersects $\partial H_{i+1}$ non-trivially. Since $D \cap \partial H_{i+1}$ is a single loop or arc, this implies that $D \cap X_{i+1}$ is a single component. 

Any arc of $\partial D \cap V_i$ with both endpoints in $L_i \times \{t_i\}$ can be isotoped out of $V_i$, so we can assume $\partial D \cap V_i$ consists entirely of vertical arcs. Since $D \cap (F_{i+1} \times \{t_{i+1}\})$ contains at most one arc, $D \cap (F_i \times \{t_i\})$ must also contain at most one arc or loop, though this arc or loop may intersect the connecting surface $F'_i$ in multiple components. One of these arcs will bound an outermost disk $E \subset D \cap H_i$ that intersects $\partial H_i$ in a single component. However, because $X_{i+1}$ is an inside block, $X_i$ must be an outside block so the inductive hypothesis implies that no such disk exists. Thus $D \cap X_{i+1}$ must be all of $D$. In this case, $D$ is a compressing disk for $X_{i+1}$ that intersects $F_{i+1} \times \{t_{i+1}\}$ in a single arc and is disjoint from $F'_i$. Such a disk must be isotopic to a vertical disk $\alpha \times [t_i,t_{i+1}]$ so $D$ satisfies the conclusions of the Lemma. 

Next, consider the case when $X_{i+1}$ is an outside block. For sufficiently small $\epsilon$, the intersection $D \cap (F_{i+1} \times [t_{i+1} - \epsilon, t_{i+1}])$ is either empty or a vertical band of the form $\alpha \times [t_{i+1} - \epsilon, t_{i+1}]$, where $\alpha$ is a loop or arc in $F_{i+1}$. Because $S_{i+1}$ is disjoint from the interior of $F_{i+1} \times \{t_i\}$,  there is an ambient isotopy of $X$ that takes $F_{i+1} \times [t_{i+1} - \epsilon, t_{i+1}]$ onto $X_{i+1}$ and sends the rest of $X_{i+1}$ into $H_i$. If we replace $D$ with its image after this isotopy, then $D$ will intersect $X_{i+1}$ in a vertical disk or annulus $\alpha \times [t_i, t_{i+1}]$, where $\alpha$ is an essential, properly embedded loop or arc in $F_{i+1}$. 

Define $X' = X \cap H_i$ and let $D' = D \cap X'$. Note that $\partial D'$ may intersect $\partial H_i$ in multiple arcs, all but one of which are contained in $F_i \setminus F'_i$. The remaining arc will intersect $\ell_{i+1}$ in two points, namely the endpoints of $\alpha$. We can make $D'$ transverse to $F'_{i-1}$ and isotope it so that $D' \cap F'_{i-1}$ is a collection of properly embedded loops and arcs in the two surfaces. If $D' \cap F'_{i-1}$ is non-empty, let $E$ be an innermost disk bounded by a loop of $D' \cap F'_{i-1}$ or an outermost disk bounded by an arc of $D' \cap F'_{i-1}$. Otherwise, let $E = D'$. 

If $E$ is contained in $H_{i-1}$ then it intersects $\partial H_{i-1}$ in a single component. However, since $X_i$ is an inside block, $X_{i-1}$ must be an outside block so the inductive hypothesis (or Lemma~\ref{lem:newsurfincomp1} in the case when $i = 2$) rules this out. 

If $E$ is contained in $X_i = F_i \times [t_{i-1}, t_i]$ then it is a properly embedded disk such that $\partial D' \cap (F_i \times \{t_{i-1}\})$ is either disjoint from $\ell_{i-1} \times \{t_{i-1}\}$ or intersects this loop in two points. If $\partial D' \cap (F_i \times \{t_{i-1}\})$ is a single arc $\alpha$ then $D'$ is isotopic to $\alpha \times [t_{i-1}, t_i]$, where $\alpha$ is intersects $\ell_{i-1}$ in at most two points, so $d_{F_i}(\alpha, \ell_{i-1}) \leq 2$. Because the arc $\alpha$ is isotopic to an arc of $\partial D'$, it intersects $\ell_{i+1}$ in at most two points, so $d_{F_i}(\alpha, \ell_{i+1}) \leq 2$ and by the triangle inequality, $d_{F_i}(\ell_{i-1}, \ell_{i+1}) \leq 4 < k$, which again contradicts our initial assumption.

Otherwise, if $\partial D \cap (F_i \times \{t_i\})$ consists of multiple arcs then at least one of these arcs will be disjoint from $\ell_{i+1} \times \{t_{i+1}\}$ and we will let $\alpha$ be one such arc. Similarly, let $\beta$ be an arc of $\partial D \cap (F_i \times \{t_{i-1}\})$ that is disjoint from $\ell_{i-1}$. Because $\alpha$ and $\beta$ are arcs in the boundary of a properly embedded disk in $F_i \times [t_{i-1}, t_i]$, Lemma~\ref{lem:masurschleimer} implies that $d_{F_i}(\alpha, \beta) < 6$. Thus the triangle inequality implies that $d_{F_i}(\ell_{i-1}, \ell_{i+1}) \leq 8 < k$. This final contradiction rules out the case when $X_{i+1}$ is an outside block, so we conclude that $X_{i+1}$ is an inside block and $D$ is a vertical disk.
\end{proof}

\begin{Coro}
\label{coro:incompsurf}
If $\ell_0,\ldots,\ell_d$ is a $(d,k)$ flexipath for $k > 8$ and $d \geq 4$ then the flat surface $S$ defined by the flexipath is incompressible in $M$.
\end{Coro}

\begin{proof}
Assume for contradiction there is a compressing disk $D$ for the flat surface $S$. We can isotope $\partial D$ disjoint from the disks $V_0$ and $V_{d+1}$ in the handlebodies on either side of $\Sigma \times [0,1]$. We will consider two cases, depending on whether $D$ is on the side of $S$ that contains the inside block at level $d$ or the outside block at level $d$.

Let $X$ be the block at level $1$ on the side of $S$ that contains $D$ and assume $D$ has been isotoped transverse to the connecting surface $F'_1$ at level $1$. Let $E$ be an innermost disk of $D$ cut off by a loop of $D \cap F'_1$ or an outermost disk cut of by an arc of $D \cap F'_1$. If $E$ is contained in $H_1$ then Lemma~\ref{lem:newsurfincomp1} implies that $D$ is a vertical disk $\alpha \times [t_0, t_1]$ in the inside block at level $1$ such that $\alpha \cap F'_1$ is a single arc $\beta$. However, if this were the case, we could extend $D$ by the vertical disk $\beta \times [t_1, t_2]$ in the outside block at level $2$. The resulting disk would intersect $\partial H_2$ in a single arc, contradicting Lemma~\ref{lem:newsurfincomp2}. 

Thus we can assume that the disk $E$ is in the closure of complement $H'_1 = M \setminus H_1$. Note that the construction of $S$ is symmetric if we reverse the direction of the path $\ell_0,\ldots,\ell_d$, and thus start building the surface from $H^+$ instead of $H^-$. Thus the intersection $S \cap H'_1$ is a flat surface defined by a $(d,k)$ flexipath, so by repeating the argument above, Lemmas~\ref{lem:newsurfincomp1} and~\ref{lem:newsurfincomp2} again imply that no such $E$ can exist.
\end{proof}

Finally, we need to calculate the genus of $S$.

\begin{Lem}
\label{lem:genusofS}
Let $\ell_0,\dots,\ell_i$ be the first $i+1$ vertices of a flexipath in $\mathcal{C}(\Sigma)$, where $\Sigma$ is the boundary of a genus $g$ handlebody $H$. Let $S_i \subset H$ be the induced layered surface. Then $S_i$ is compact, properly embedded and two-sided. If $i$ is even then $S_i$ is a once-punctured, genus $\frac{1}{2}i$ surface. If $i$ is odd then $S_i$ is a twice-punctured, genus $\frac{1}{2}(i-1)$ surface.
\end{Lem}

\begin{proof}
We will proceed by induction on $i$. For $i = 0$, $S_0$ is a disk with boundary in $\Sigma$ isotopic to $\ell_0$. In other words, $S_0$ is a once-punctured sphere (a genus-zero surface). 

For the inductive step, assume that the Lemma holds for $i-1$. Note that we construct $S_i$ by gluing a pair-of-pants to $S_{i-1}$ along either one or two loops. If $i$ is odd, then we glue the pair-of-pants to $S_{i-1}$ along the single loop $L_{i-1} = \ell_{i-1}$, turning the once-punctured, genus $\frac{1}{2}(i-1)$ surface into a twice-punctured surface of genus $\frac{1}{2}(i-1)$. If $i$ is even then we glue the pair-of-pants to $S_{i-1}$ along the two loops $L_{i-1}$, which increases its genus by one and reduces the number of boundary components back to one. In other words, it turns the twice-punctered surface of genus $\frac{1}{2}((i-1)-1) = \frac{1}{2}(i-2)$ into a once-punctured surface of genus $\frac{1}{2}(i-2) + 1 = \frac{1}{2}i$. Thus by induction, the Lemma holds for every positive value of $i$.

Finally, note that $S$ is properly embedded, connected, compact and separating by construction. Since a handlebody is orientable and $S$ is separating, $S$ must be two-sided.
\end{proof}

\begin{proof}[Proof of Theorem~\ref{thm:main1}]
Let $\Sigma$ be a compact, connected, closed, orientable genus $g \geq 2$ surface. By Corollary~\ref{coro:flexipath}, there is a strict $(d,k)$ flexipath for every pair of values $k > 0$, $d \geq 2$, so for a given even integer $d \geq 4$, let $k = max\{d+1, 9\}$ and let $\ell_0,\ldots,\ell_d$ be a $(d,k)$ flexipath. Let $M$ be a closed, distance $k$ filling with respect to $\ell_0$ and $\ell_d$ of the path manifold $M'$ defined by $\ell_0,\ldots,\ell_d$. Then by Corollary~\ref{coro:filledist}, the induced Heegaard surface $\Sigma$ for $M$ has distance $d(\Sigma) = d$.

Let $S$ be the embedded flat surface in $M$ induced by the path $\ell_0,\ldots,\ell_d$. Then $S$ is the result of gluing an annulus into the two boundary components of the surface $S_{d-1}$ defined in Lemma~\ref{lem:genusofS}. The surface $S_{d-1}$ has genus $\frac{d-2}{2}$ so Lemma implies that $S$ is a compact, connected, closed, two-sided surface of genus $\frac{1}{2}d$. Since $\ell_0,\ldots,\ell_d$ is a $(k,d)$ flexipath with $k > 8$ and $d \geq 2$, Corollary~\ref{coro:incompsurf} implies that $S$ is an incompressible surface.
\end{proof}

\section{The alternate Heegaard splitting}
\label{sect:splitting2}

Let $\ell_0,\dots,\ell_d$ be an almost strict $(d,k)$ flexipath and let $j$ be the index such that $\ell_{j-1} \cap \ell_{j+1}$ is a single point in the once-punctured torus $F_j$. Let $S^-$ be a surface in a handlebody $H^-$ defined by the path $\ell_0,\ldots,\ell_{j-1}$. The reverse path $\ell_d,\ell_{d-1},\ldots,\ell_{j+1}$ is also a flexipath, so we can use it to construct a surface $S^+$ in a second handlebody $H^+$. By the construction of $M$, we can embed $H^-$ and $H^+$ in $M$ so that their boundaries coincide with $\Sigma$.

Note that each loop of $\partial S^- \subset \partial H^-$ is parallel to $\ell_{j-1}$ and each loop of $\partial S^+ \subset \partial H^+$ is parallel to $\ell_{j+1}$. Thus the two surfaces $S^-$ and $S^+$ do not form a single closed surface, but rather each of the two loops of $\partial S^-$ intersects each of the two loops in $\partial S^+$ in single point, for a total of four points of intersection.

Let $N(\ell_{j-1})$ and $N(\ell_{j+1})$ be the images in $\Sigma$ of the regular neighborhoods used in the final steps of constructing $S^-$ and $S^+$, so that $\partial N(\ell_{j-1}) = L_{j-1}$, etc. Then in particular, $F_{j-1}$ will be the complement in $\Sigma$ of $N(\ell_{j-1})$ and $F_{j+1}$ is the complement of $N(\ell_{j+1})$. The intersection $Q = N(\ell_j) \cap N(\ell_{j+1})$ is a disk whose boundary consists of two opposite arcs in $\partial S_{j-1}$ and two opposite arcs in $\partial S_{j+1}$. We will call $Q$ the \textit{flipped square}. The intersection $F'' = F_j \cap F_{j+1}$ is a once-punctured surface of genus $g-1$, namely the complement in $\Sigma$ of the once-punctured torus $N(\ell_j) \cup N(\ell_{j+1})$.

Define $S \subset M$ to be the union $S^- \cup S^+ \cup Q \cup F''$. The portion of this surface near the disk $Q$ is shown in Figure~\ref{fig:flippedsquare}. 
\begin{figure}[htb]
  \begin{center}
  \includegraphics[width=3.5in]{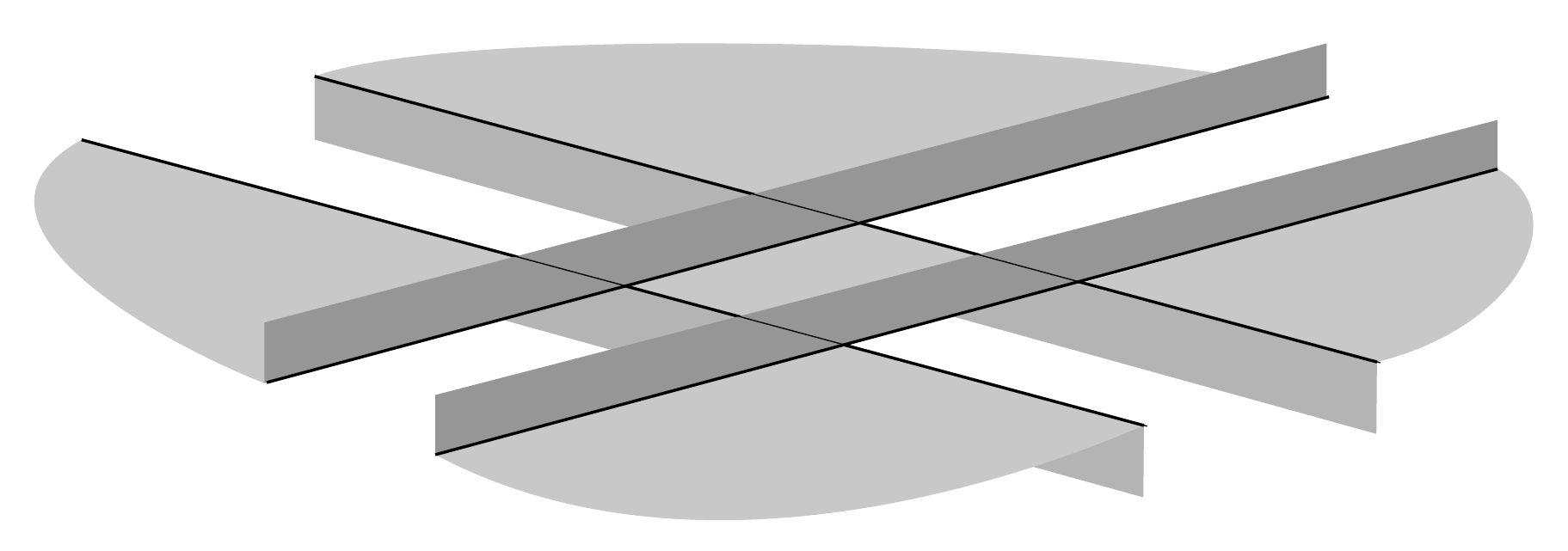}
  \caption{A flipped square defining an index-one flat surface.}
  \label{fig:flippedsquare}
  \end{center}
\end{figure}

We can check that $S$ is a surface by noting that each surface $S^\pm$ meets each of $Q$ and $F''$ in an arc. The endpoints of these eight arcs are four points in $\Sigma$ and in a regular neighborhood of each point, the four components of the union meet to form a disk. 

\begin{Def}
We will say that the surface $S \subset M$ is the \textit{index-one flat surface} defined by the almost strict flexipath $\ell_0,\ldots,\ell_d$.
\end{Def}

The terminology here follows that in~\cite{me:upperbnd}, where such surfaces arise from a different, but very much complementary, construction.

An alternate view of $Q \cup F''$, without the adjacent vertical annuli is shown in blue in Figure~\ref{fig:globalflipped}. Notice that a regular neighborhood in $\Sigma$ of $Q \cup F''$ is a twice-punctured surface with the same genus as $\Sigma$. (In the example shown in the Figure, it's a twice-punctured genus-two surface.) The union of $Q \cup F''$ with the adjacent vertical annuli deformation retracts on to $Q \cup F''$, and thus has the same Euler characteristic, but has four boundary components instead of two. The reader can check that this implies that the union of $Q \cup F''$ with the four adjacent vertical annuli is a four-times punctured surface of genus $g-1$ (where $g$ is the genus of $\Sigma$.)
\begin{figure}[htb]
  \begin{center}
  \includegraphics[width=3.5in]{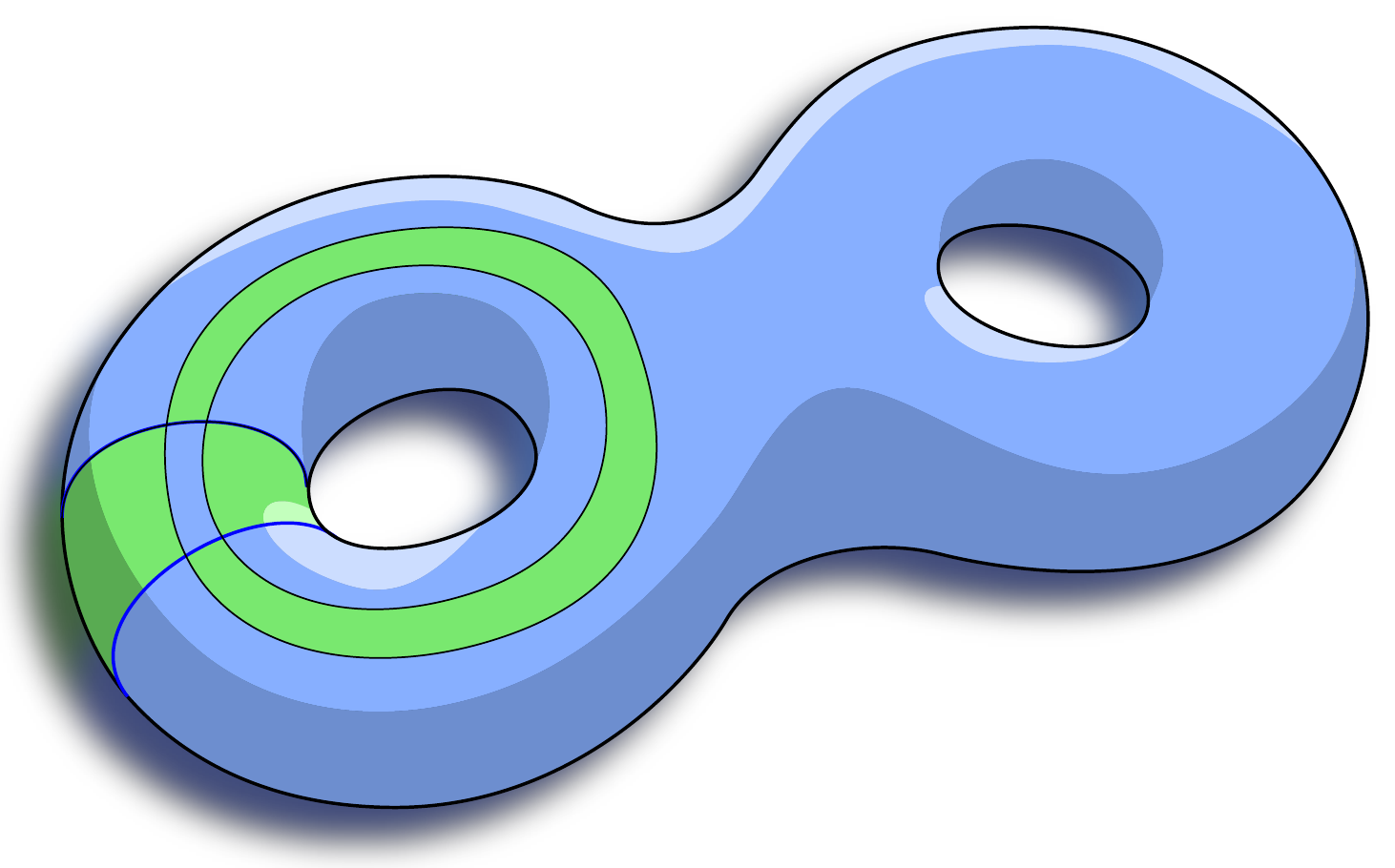}
  \caption{The horizontal part of an index-one flat surface near a flipped square, within a genus-two surface $\Sigma$.}
  \label{fig:globalflipped}
  \end{center}
\end{figure}

Let $A_-$, $B_-$ be the closures of the components of $H^- \setminus S^-$ and let $A_+$, $B_+$ be the closures of the components of $H^+ \setminus S^+$ such that $A_-$ and $A_+$ contain the annuli $N(\ell_{j-1})$ and $N(\ell_{j+1})$, respectively. Then the surface $S$ separates $M$ into components $H^-_S = A_- \cup B_+$ and $H^+_S = A_+ \cup B_-$. By Corollary~\ref{lem:newsurfacesincomp}, each of $A_\pm$ and $B_\pm$ is a handlebody. The handlebodies $A_-$ and $B_+$ meet along the disk $N(\ell_{j-1}) \setminus N(\ell_{j+1})$ so their union $H^-_S$ is a handlebody. Similarly, $A_+$ and $B_-$ meet along $N(\ell_{j+1}) \setminus N(\ell_{j-1})$, forming the second handlebody $H^+_S$. Thus we have the following:

\begin{Lem}
\label{lem:itsheegaard}
The triple $(S, H^-_S, H^+_S)$ is a Heegaard splitting of $M$.
\end{Lem}

A Heegaard splitting is said to be \textit{strongly irreducible} if every pair of compressing disks on opposite sides of the Heegaard surface intersect non-trivially. It is straightforward to check that a stabilized Heegaard splitting is not strongly reducible (i.e.\ \textit{weakly reducible}), so in order to check that $S$ is not a stabilization of $\Sigma$, we will show that $S$ is strongly irreducible.

\begin{Lem}
\label{lem:alternatesi}
If $S$ is the Heegaard surface induced by a $(d, k)$ flexipath such that $d \geq 4$ is even and $k > 8$ then $S$ is strongly irreducible.
\end{Lem}

\begin{proof}
Assume for contradiction there are disjoint compressing disks $D^- \subset H^-_S$, $D^+ \subset H^+_S$. Note that the two loops $\partial S^-$ and the two loops $\partial S^+$ are essential, no component of $\partial S^-$ is parallel in $S$ to a component of $S^+$, and any two of these loops intersect minimally in $S$. Isotope $\partial D^-$ and $\partial D^+$ within $S$ so as to minimize their intersections with these four loop. We can do this, for example, by choosing an abstract hyperbolic metric on $S$ in which the loops $\partial S^-$ and $\partial S^+$ are geodesics. If we then isotope $\partial D^-$ and $\partial D^+$ to geodesics, we will ensure that they remain disjoint, but intersect $\partial S^- \cup \partial S^+$ minimally. We can extend the isotopies of $\partial D^-$ and $\partial D^+$ to isotopies of the disks $D^-$ and $D^+$ within $H^-$ and $H^+$, respectively. 

As above, let $Q = N(\ell_{j-1}) \cap N(\ell_{j+1})$ be the flipped square in $S$. Let $Q^+$ be the separating compressing disk $N(\ell_{j+1}) \setminus N(\ell_{j-1})$ for $H^+_S$. This appears in Figure~\ref{fig:flippedsquare} as two of the white regions adjacent to the flipped square in the middle. (The two regions connect up outside the region shown in the Figure to form a disk.) Both $D^+$ and $Q^+$ are contained in $A_+ \cup B_-$. By construction, the boundaries of $D^+$ and $Q^+$ will be transverse and we can isotope the interior of $D^+$ to be transverse to $Q^+$, such that $D^+ \cap Q^+$ is a collection of arcs. If $D^+$ is disjoint from $Q^+$ then define $E = D^+$. Otherwise, let $E$ be an outermost disk cut off by an arc of $D^+ \cap Q^+$. 

Since $Q^+$ separates $A_+$ from $B_-$, the disk $E$ is contained in either $A_+$ or $B_-$. If $E$ is contained in $B_-$, let $X_{j-1}$ be the block of $B_-$ at level $j-1$. Let $F'_{j-2}$ be the horizontal subsurface between $X_{j-1}$ and the block at level $j-2$. Isotope $E$ (in the complement of a regular neighborhood of $F_{j-1}$) to be transverse to $F'_{j-2}$ and so that the intersection is minimal. 

Let $E'$ be an outermost disk of $E$ cut off by an arc of $E \cap F'_{j-2}$, or an innermost disk bounded by a loop of $E \cap F'_{j-2}$. Because of the way we chose the labels on $A_-$ and $B_-$, the block $X_{j-1}$ is an inside block, so the block at level $j-2$ is an outside block. (Note that this block will be a component of $H_0^- \setminus S_0$ if $j = 2$.) Thus Lemma~\ref{lem:newsurfincomp2} implies that $E'$ must be contained in $X_{j-1}$. The boundary of $E$ intersects $Q^+$ in at most one arc (since $E'$ is contained in $E$) and intersects $\ell_{j-2}$ in at most arc, since $E'$ is outermost in $E \setminus F'_{j-2}$. Applying the Masur-Schleimer Lemma (Lemma~\ref{lem:masurschleimer}) and the triangle inequality, as in the proof of Lemma~\ref{lem:newsurfincomp2}, implies that $d_{F_{j-1}}(\ell_{j+1}, \ell_{j-2}) < 8$. This contradicts the assumption that $\ell_0,\ldots,\ell_d$ is a $(d,k)$ flexipath for $k > 8$, so we conclude that $E$ must be contained in $A_+$.

In this case, note that the block of $A_+$ adjacent to $\partial H^+$ is an outside block. Thus Lemma~\ref{lem:newsurfincomp2} implies that $E \cap \partial H^+$ contains at least two arcs, since there are no disks with just one arc in $\partial H^+$. All these arcs are in $N(\ell_{j+1})$, and at most one of the arcs intersects $Q^+$, so some arc in $\partial E$ must run across the quadrilateral $Q$, parallel to the arcs of $\ell_{j-1}$, as indicated by the red portion of a disk above $Q$ in Figure~\ref{fig:disksintersect}.
\begin{figure}[htb]
  \begin{center}
  \includegraphics[width=3.5in]{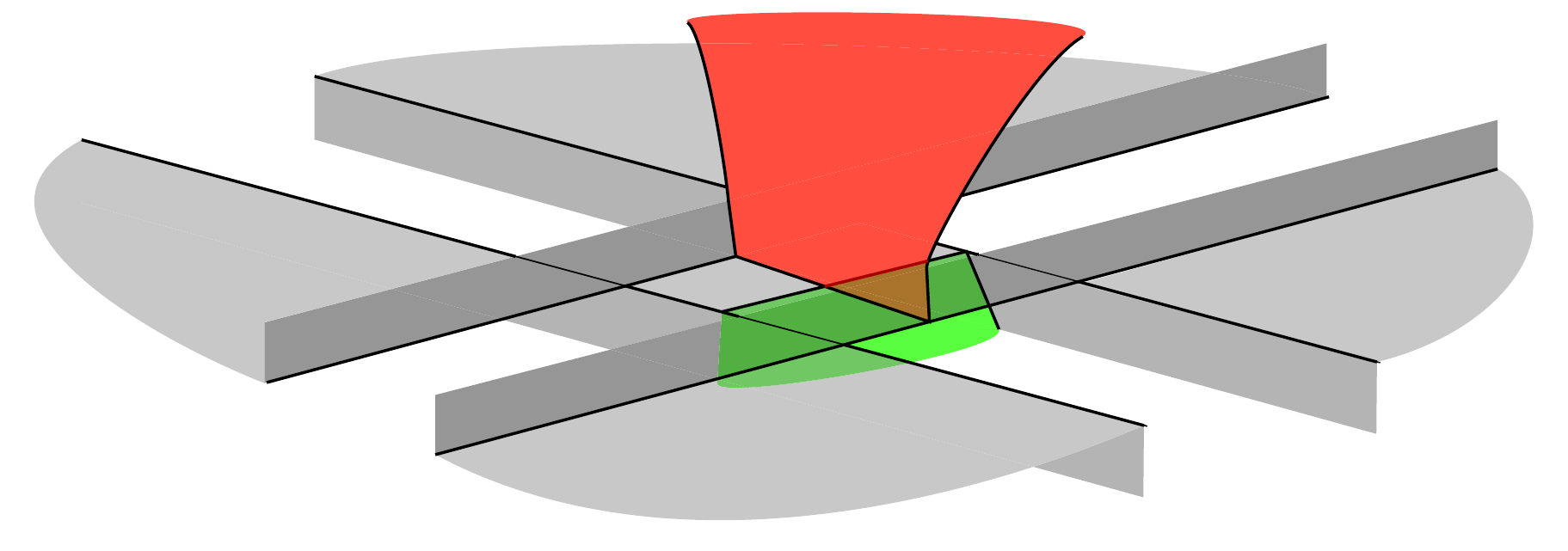}
  \caption{Local pictures of $D^-$ and $D^+$}
  \label{fig:disksintersect}
  \end{center}
\end{figure}

We can apply a similar argument to $A_- \cup B_+$. Note that in the case when $E$ is contained in $A_-$, the bottom block in $A_-$ is at level $j + 1$ and the adjacent outside block is at level $j+2$. Therefore, we must choose $j + 2 \leq d$, or equivalently $j \leq d-2$. Thus we need $d \geq 4$ so that we can choose $2 \leq j \leq d-2$. 

The argument for $A_- \cup B_+$ implies that an arc of $\partial D^-$ must cross $Q$ parallel to $\ell_{j+1}$. However, because the loops $\ell_{j-1}$ and $\ell_{j+1}$ intersect in a point of $Q$, each arc of $\partial D^-$ in $Q$ must intersect each arc of $\partial D^+$ in $Q$. This contradicts the assumption that $\partial D^-$ and $\partial D^+$ are disjoint, so we conclude that $S$ is a strongly irreducible Heegaard surface.
\end{proof}

Note that $S$ has distance exactly two, since the disks $Q^+$ and $Q^-$ are compressing disks on opposite sides of $S$ and the complement of their union is the interior of the surfaces $S^-$ and $S^+$, each of which contains many essential loops. This immediately implies that $S$ is not isotopic to $\Sigma$, since $d(\Sigma) \geq 6$. We can also show they are distinct by calculating the genus of $S$, which we need to do for the proof of Theorem~\ref{thm:main2} anyway.

\begin{Lem}
\label{lem:alternategenus}
An index-one flat surface $S$ induced by a almost strict flexipath $\ell_0,\dots,\ell_d$ has genus $\frac{1}{2}d + g - 1$.
\end{Lem}

\begin{proof}
By Lemma~\ref{lem:itsheegaard}, $S$ is a Heegaard surface, and we will calculate the genus of each of the handlebodies that make up its complement. As above, the surface $S^-$ separates $H^-$ into components $A^-$, $B^-$ where $A^-$ contains $N(\ell_{j-1}) \subset \partial H^-$ and $B^-$ contains $F_{j-1} \subset \partial H^-$. Thus the boundary of $A^-$ is the union of $S^-$ and the annulus $N(\ell_{j-1})$. By assumption, $j$ is even so $S^-$ has genus $\frac{1}{2}(j - 2)$ by Lemma~\ref{lem:genusofS}. When we glue the annulus between the two boundary components of $S^-$, we find that $\partial A^-$ has genus $\frac{1}{2}j$. On the other hand, the boundary of $B^-$ is the union of $S^-$ with the genus $g-1$ surface $F_i$ along two loops, and thus has genus $\frac{1}{2}j + (g-1)$.

We can calculate genera of the components of $H^+ \setminus S_{P_+}$ similarly. The path $\ell_d,\ell_{d-1},\ldots,\ell_{j}$ has length $d - j$. (Recall that the length of the path is the number of edges, which is one less than the number of vertices.) Thus by the above argument, $\partial A^+$ has genus $\frac{1}{2}(d-j)$ and $\partial B^+$ has genus $\frac{1}{2}(d-j) + (g-1)$. When we glue two handlebodies along a disk, the genus of the resulting handlebody is the sum of the original genera. Thus we find that $\partial H^-_S$ has genus $\frac{1}{2}j + \frac{1}{2}(d-j) + (g-1) = \frac{1}{2}d + (g-1)$. Similarly, $\partial H^+_S$ has genus $\frac{1}{2}j + (g-1) + \frac{1}{2}(d-j) = \frac{1}{2}d + (g-1)$. As expected, the two genera agree and we find that the Heegaard surface $S$ has genus $\frac{1}{2}d + (g-1)$.
\end{proof}

An alternate way to calculate the genus of $S$ is to count the number of pairs-of-pants used to construct it. As indicated schematically in Figure~\ref{fig:frombothsides2}, there is again one pair-of-pants for each edge in the path $\ell_0,\ldots,\ell_d$ with four exceptions: The first and last edges correspond to the annuli on the left and right sides of the schematic, and the two edges adjacent to the flipped square correspond to a genus $g-1$ subsurface with four punctures, rather than a four-punctured sphere composed two pairs-of-pants.
\begin{figure}[htb]
  \begin{center}
  \includegraphics[width=4.5in]{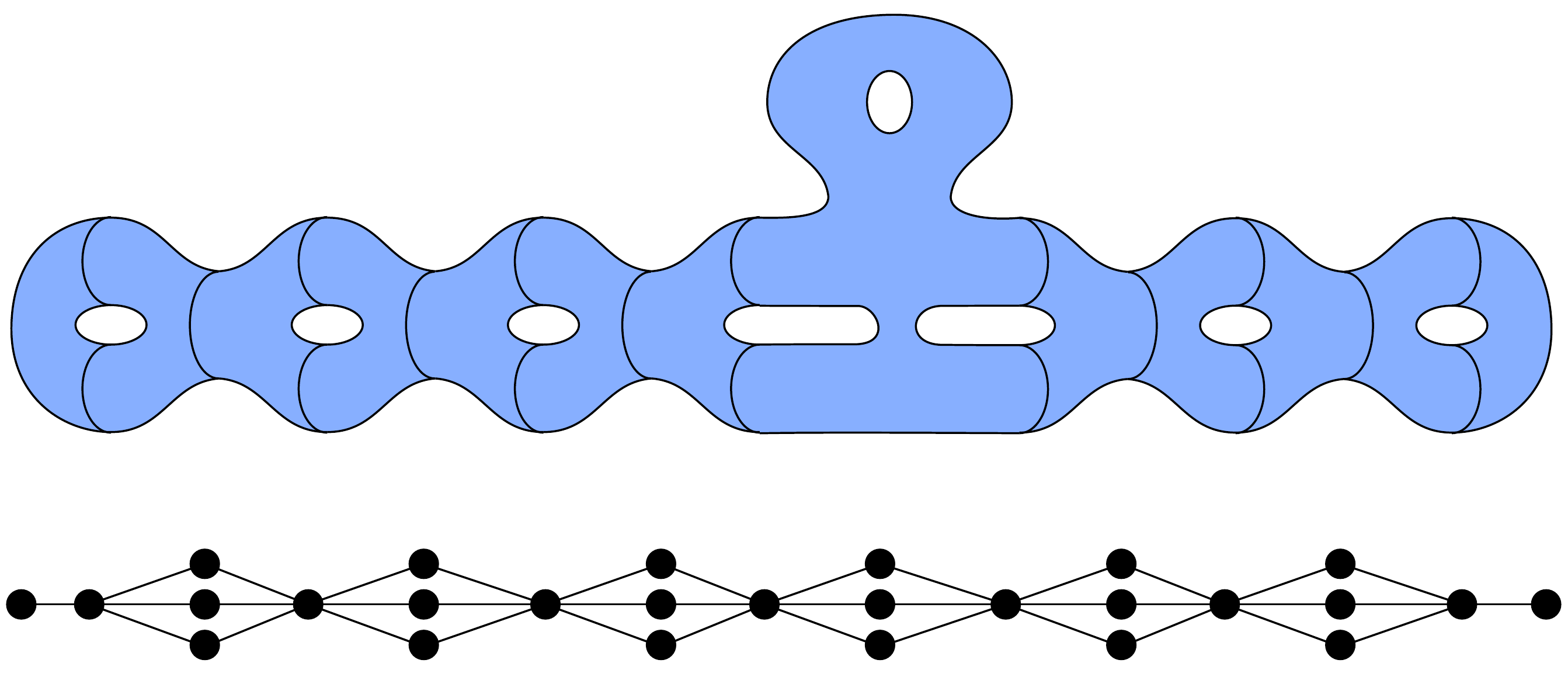}
  \caption{A schematic picture of the alternate Heegaard surface, along side the path used to define it.}
  \label{fig:frombothsides2}
  \end{center}
\end{figure}

\begin{proof}[Proof of Theorem~\ref{thm:main2}]
Let $\Sigma$ be a compact, connected, closed, orientable genus $g \geq 2$ surface. By Corollary~\ref{coro:flexipath}, there is an almost strict $(d,k)$ flexipath for every pair of values $k > 0$, $d \geq 2$ (with $d$ even), so for a given even integer $d \geq 4$, let $k = max\{d+1, 9\}$ and let $\ell_0,\ldots,\ell_d$ be a $(d,k)$ flexipath. Let $M$ be a closed, distance $k$ filling with respect to $\ell_0$ and $\ell_d$ of the path manifold $M'$ defined by $\ell_0,\ldots,\ell_d$. Then by Corollary~\ref{coro:filledist}, the induced Heegaard surface $\Sigma$ for $M$ has distance $d(\Sigma) = d$.

Let $S$ be the index-one flat surface in $M$ induced by the path $\ell_0,\ldots,\ell_d$. By Lemma~\ref{lem:alternategenus}, $S$ is a compact, connected, closed, two-sided surface of genus $\frac{1}{2}d + (g-1)$. By Lemma~\ref{lem:itsheegaard}, the surface $S$ determines a Heegaard splitting $(S, H^-_S, H^+_S)$ for $M$. Since $\ell_0,\ldots,\ell_d$ is an almost strict $(d,k)$ flexipath with $k > 8$ and $d \geq 4$, Lemma~\ref{lem:alternatesi} implies that $S$ is strongly irreducible. Since the genus of $S$ is strictly greater than that of $\Sigma$, the two surfaces cannot be isotopic. Moreover, because $S$ is strongly irreducible, it is not stabilized, so in particular $S$ is not a stabilization of $\Sigma$.
\end{proof}

\bibliographystyle{amsplain}
\bibliography{highdist}

\end{document}